%% file: faces.tex
\documentclass[letter]{amsart}
\input{defs.tex}

\begin{document}

\author[Olivier Bernardi]{Olivier Bernardi}
\thanks{Department of Mathematics, Brandeis University, Waltham MA, USA.
bernardi@brandeis.edu.\\% Supported by NSF grant DMS-1308441 and DMS-1800681.\\
}

\title{Bijections for faces of braid-type arrangements}
\date{\today}

\begin{abstract}
We establish a general bijective framework for encoding faces of some classical hyperplane arrangements.
%certain hyperplane arrangements of Coxeter type.
Precisely, we consider hyperplane arrangements in $\RR^n$ whose hyperplanes are all of the form $\{x_i-x_j=s\}$ for some $i,j\in[n]$ and $s\in \ZZ$.
Such an arrangement $\mA$ is \emph{strongly transitive} if it satisfies the following condition: if $\{x_i-x_j=s\}\notin \mA$ and $\{x_j-x_k=t\}\notin \mA$ for some $i,j,k\in [n]$ and $s,t\geq 0$, then $\{x_i-x_k=s+t\}\notin \mA$.
For any strongly transitive arrangement $\mA$, we establish a bijection between the faces of $\mA$ and some set of decorated plane trees. 
\end{abstract}

\maketitle

\section{Introduction}\label{sec:intro}
%
%
%Given a set of integers $S\subseteq \ZZ$ we define, for every dimension $n>0$, the arrangement $\mA_S^n\subseteq \mA_m^n\subset \RR^n$ as follows:
%$$\mA_S^n:=\bigcup_{\substack{1\leq i<j\leq n\\ s\in S}}\{x_i-x_j=s\}.$$

In this article we establish some bijective results about the faces of some classical families of hyperplane arrangements. 
Specifically, we consider real hyperplane arrangements made of a finite number of hyperplanes of the form 
\begin{equation}\label{eq:hyperplane}
\{(x_1,\ldots,x_n)\in \RR^n ~\mid~ x_i-x_j=s\},
\end{equation} 
with $i,j\in\{1,\ldots,n\}$ and $s\in \ZZ$. We call them \emph{braid-type arrangements}. From now on, we make an abuse of notation and denote by $\{x_i-x_j=s\}$ the hyperplane in~\eqref{eq:hyperplane}.

Given a set of integers $S\subseteq \ZZ$ we define, for every dimension $n>0$, the braid-type arrangement $\mA_S^n\subset \RR^n$ as follows:
$$\mA_S^n:=\bigcup_{\substack{1\leq i<j\leq n\\ s\in S}}\{x_i-x_j=s\}.$$
Classical examples include the \emph{braid}, \emph{Catalan}, \emph{Shi}, \emph{semiorder}, and \emph{Linial} arrangements, which correspond to $S=\{0\},~\{-1,0,1\},~\{0,1\},~\{-1,1\}$, and $\{1\}$ respectively. These arrangements are represented
in Figure~\ref{fig:classical-arrangements}.

% OB: Maybe place the above sentence at an other place.
%; Important early references are \cite{Orlik:hyperplane-arrangements,Orlik:Poincare-hyperplanes,Zaslavsky:region-hyperplanes}.

\fig{width=\linewidth}{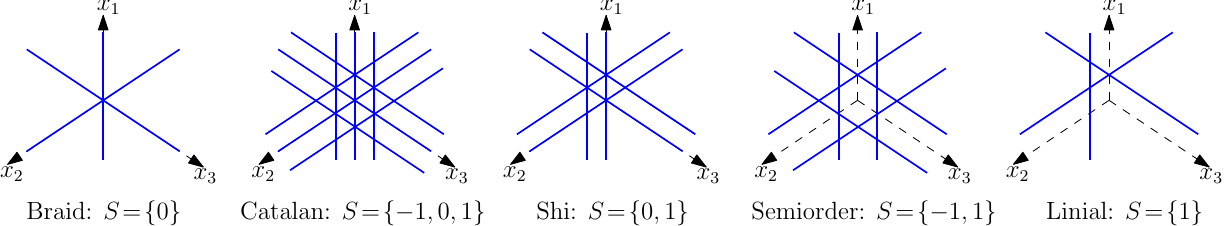}{The braid, Catalan, Shi, semiorder, and Linial arrangements in dimension $n=3$. Here and later, the braid-type arrangements in dimension 3 are represented by drawing their intersection with the hyperplane $\{(x_1,x_2,x_3)\mid x_1+x_2+x_3=0\}$.}
%Note that these classical arrangements correspond to all the sets $S\subseteq\{-1,0,1\}$ up to symmetry. 

There is an extensive literature on counting regions of braid-type arrangements, starting with the work of Shi \cite{Shi:nb-regions,Shi:nb-regions-Weyl}. We refer the reader to~\cite{Orlik:hyperplane-arrangements} or \cite{Stanley:hyperplane-arrangements} for an introduction to the general theory of hyperplane arrangements. Seminal counting results about braid-type arrangements were established by Stanley~\cite{Stanley:hyperplane-tree-inversions,Stanley:hyperplane-interval-orders-overview}, Postnikov and Stanley~\cite{Postnikov:coxeter-hyperplanes}, and Athanasiadis~\cite{Athanasiadis:finite-field-method,Athanasiadis:PhD}. Since then, the subject has become quite popular among combinatorialists, and many beautiful counting formulas and bijections were discovered for various families of arrangements; see in particular \cite{Ardila:Tutte-hyperplanes,Armstrong:Shi-Ish,Athanasiadis:free-deformations,Athanasiadis:Catalan-Hyperplanes,Athanasiadis:free-deformations,Athanasiadis:bijection-Shi,Corteel:hyperplane-valued-graphs,Gessel:labeled-trees-Shur-positive,Headley:hyperplane-Weyl-groups,Hopkins:bigraphical-arrangements,Hopkins:G-Parking-labelling,Leven:bijections-Ish-Shi}.% (see also Section~\ref{sec:definitions} for additional references). 

By contrast, there is a paucity of results about lower dimensional faces in braid-type arrangements. Notable exceptions are enumerative formulas by Athanasiadis~\cite{Athanasiadis:finite-field-method}, and bijections by Levear~\cite{Levear} about the Catalan and Shi arrangements.

In this article we develop a general framework for bijectively encoding the faces of a large class of braid-type arrangements. Our framework applies to all \emph{strongly transitive} arrangements, which are the braid-type arrangements $\mA$ satisfying the following condition: if $\{x_i-x_j=s\}\notin \mA$ and $\{x_j-x_k=t\}\notin \mA$ for some distinct indices $i,j,k$ and $s,t\in \NN$, then $\{x_i-x_k=s+t\}\notin \mA$. For any such arrangement $\mA\subseteq \RR^n$ we give an explicit bijection between the set of faces of $\mA$ and a set of trees with $n$ labeled nodes and some marked edges (with the number of marked edges corresponding to the codimension of the associated face of $\mA$). This applies in particular to the braid, Catalan, Shi and semiorder arrangements.
For the case of the Catalan and Shi arrangements, our bijection coincides with the one established by Levear~\cite{Levear}, up to minor changes in notation. 

The bijections for faces of braid-type arrangements presented here are an extension of a general bijective framework for regions of (transitive) arrangements that we established in \cite{OB}. We actually use the results from \cite{OB} to establish those in the present article.

The article is organized as follows.
In Section~\ref{sec:background} we set our notation, and recall some necessary background from \cite{OB}. 
In Section~\ref{sec:main} we state our main result, which gives a bijection for faces of strongly transitive arrangements.
In Section~\ref{sec:exp} we apply our bijection to a few classes of strongly transitive arrangements. In particular we discuss a multivariate generalization of the Catalan arrangement, as well as a family of arrangements interpolating between the Catalan and Shi arrangements. 
In Section~\ref{sec:GF} we derive the generating function for the faces of the Catalan and semiorder arrangements, and discuss a more general setting under which the generating function of faces satisfies a polynomial equation.
The proof of our main results is given in Section~\ref{sec:proof}.

%%Given a set of integers $S\subseteq \ZZ$ we define, for every dimension $n>0$, the arrangement $\mA_S^n\subseteq \mA_m^n\subset \RR^n$ as follows:
%%$$\mA_S^n:=\bigcup_{\substack{1\leq i<j\leq n\\ s\in S}}\{x_i-x_j=s\}.$$
%%\OB{Define the Catalan, Shi, semiorder and Linial arrangements.}

\section{Notation and background}\label{sec:background}
In this section we set our notation and recall some relevant results from \cite{OB}.

%We shall use the following notation. 
We denote by $\NN=\{0,1,\ldots\}$ the set of non-negative integers.
For integers $m,n$, we define $[m;n]:=\{k\in\ZZ\mid m\leq k\leq n\}$, and $[n]:=[1;n]$.
For a positive integer $n$, we denote by $\fS_n$ the set of permutations of $[n]=\{1,\ldots,n\}$. 
%For any integers $m<n$ we define $[m:n]:=\{k\in \ZZ\mid m\leq k\leq n\}$. 

For a real hyperplane arrangement $\mA$, we denote by $\mR(\mA)$ the set of regions of $\mA$.

\begin{definition}
Let $m,n\in \NN$. The \emph{$m$-Catalan arrangement in dimension $n$} is
$$\mA_m^n=\bigcup_{\substack{1\leq i<j\leq n\\ s\in[-m;m]}}\{x_i-x_j=s\}\subseteq\RR^n.$$
\end{definition}

We will now introduce a canonical way to label the hyperplanes in braid-type arrangements.

\begin{notation}
For $m,n\in \NN$ we define 
$$\trmn:=\{(i,j,s)\mid i,j\in [n],~s\in[0;m]\textrm{ such that }i\neq j \textrm{ and }(s>0\textrm{ or }i>j)\}.$$
\end{notation}
Observe that each hyperplane of $\mA_m^n$ is of the form $\{x_i-x_j=s\}$ for exactly one triple $(i,j,s)$ in $\trmn$. In particular,
$$\mA_m^n=\bigcup_{(i,j,s)\in\trmn}\{x_i-x_j=s\}.$$
For an arrangement $\mA\subseteq \mA_m^n$ we define 
$$\trmn(\mA):=\{(i,j,s)\in\trmn\mid \{x_i-x_j=s\}\in \mA\},$$
so that 
$$\ds \mA=\bigcup_{(i,j,s)\in \trmn(\mA)}\{x_i-x_j=s\}.$$
% $s\in[0;m]$ and $i,j\in [n]$ with either $s>0$ or $i>j$. In \cite{OB} bijections were established for the regions of \emph{transitive arrangements}.

\begin{definition}
%An arrangement $\mA\subseteq \mA_m^n$ is \emph{transitive} if for all distinct indices $i,j,k\in[n]$ and integers $s,t\geq 0$ one has: \\
%if $\{x_i-x_j=s\}\notin \mA$ with $s>0$ or $i>j$, and $\{x_j-x_k=t\}\notin \mA$ with $t>0$ or $j>k$, then $\{x_i-x_k=s+t\}\notin \mA$.\\
An arrangement $\mA\subseteq \mA_m^n$ is \emph{transitive} if for all distinct indices $i,j,k\in[n]$ and integers $s,t\geq 0$ such that $(i,j,s)$ and $(j,k,t)$ are in $\trmn$ the following holds: \\
\centerline{if $\{x_i-x_j=s\}\notin \mA$ and $\{x_j-x_k=t\}\notin \mA$, then $\{x_i-x_k=s+t\}\notin \mA$.}
\end{definition}

\begin{definition}
A \emph{$(m,n)$-tree} is a rooted $(m+1)$-ary tree with $n$ nodes labeled with distinct labels in $[n]$ (the leaves have no labels). 
We denote by $\mT_m^n$ the set of all $(m,n)$-trees (there are $\frac{n!}{mn+1}{(m+1)n \choose n}$ of them). 

%(It is easy to see that $|\mT_m^n|=\frac{((m+1)n)!}{(mn+1)!}$).
For a tree $T\in\mT_m^n$ we identify the nodes with their labels in $[n]$ (so that the node set of $T$ is~$[n]$).
By definition, a node $j\in [n]$ of $T$ has exactly $m+1$ (ordered) children, which are denoted by $0\child(j)$, $1\child(j)$, \ldots, $m\child(j)$ respectively.

The node $i\in[n]$ is the \emph{$s$-cadet}\footnote{Here the term \emph{cadet} is meant as ``youngest surviving child''.} of the node $j\in[n]$ in $T$ if $i=s\child(j)$ and $t\child(j)$ is a leaf for all $t\in[s+1;m]$. In this case, we write $$i=s\cadet(j),$$
 and we say that $\{i,j\}$ is a \emph{cadet edge} of $T$.
\end{definition}

\begin{definition}
Let $\mA\subseteq \mA_m^n$ be an arrangement. We define 
$$\Tmn(\mA):=\{T\in\mT_m^n\mid \forall (i,j,s)\in\trmn\setminus\trmn(\mA) ,~i\neq s\cadet(j)\}.$$
\end{definition}
It was established in \cite{OB} that, for any transitive arrangement $\mA$, the regions of $\mA$ are in bijection with the trees in $\Tmn(\mA)$.
In order to describe the bijection we need to introduce a total order on the vertices of each tree in $\mT_m^n$. 

\begin{definition}
For every vertex $v$ of in a tree $T$, we consider the sequence of vertices $v_0,v_1,\ldots,v_k=v$ on the path of $T$ from the root $v_0$ to the vertex $v$, and define 
$$\pathT(v):=(s_1,\ldots,s_k),$$ 
where $v_i=s_i\child(v_{i-1})$ for all $i\in[k]$. 
We also define the \emph{drift} of $v$ as $$\driftT(v):=\sum_{i=1}^k s_i.$$
\end{definition}

\begin{definition}
Let $T$ be a tree in $\mT_m^n$. We define a total order $\preceqT$ on the vertices of $T$ as follows.
Let $v,w$ be distinct vertices of $T$. Let $\pathT(v)=(s_1,\ldots,s_k)$ and $\pathT(w)=(t_1,\ldots,t_\ell)$. 
Then $v\precT w$ if either 
\bitem
\item $\driftT(v)<\driftT(w)$, or 
\item $\driftT(v)=\driftT(w)$ and there exists $j\leq k$ such that $(s_1,\ldots,s_j)=(t_1,\ldots,t_j)$ and ($j=k$ or $s_{j+1}>t_{j+1}$).
\eitem
\end{definition}
The order $\precT$ is indicated for the tree represented in Figure~\ref{fig:prec-order}.

\fig{width=.25\linewidth}{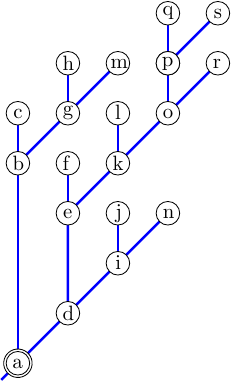}{The order $\preceq_T$ for a binary tree $T$. The vertices of $T$ are ordered as follows: $a\precT b\precT c\precT \cdots \precT s$. Here the horizontal placement of vertices corresponds to their drift in the tree. Hence, in this representation, the relation $v\prec_T w$ can be thought as meaning that either ``$v$ is strictly to the left of $w$'' in the tree $T$ or ``$v$ is at the vertical of $w$ but below $w$'' in the tree~$T$.}

\begin{definition}
Let $\mA\subseteq \mA_m^n$ be an arrangement. For every tree $T\in \Tmn(\mA)$ we define the polyhedron $\Phi_A(T)\subseteq \RR^n$ as follows:
$$\phi_\mA(T)= \left(\bigcap_{\substack{(i,j,s)\in \trmn(\mA) \\ i\precT s\child(j)}}\{x_i-x_j<s\}\right)\cap \left(\bigcap_{\substack{(i,j,s)\in \trmn(\mA)\\ i\succeqT s\child(j)}}\{x_i-x_j>s\}\right).$$
\end{definition}
It is clear from the definition that $\phi_\mA(T)$ is either empty or a region of $\mA$. In fact, it is always a region of $\mA$, and the following result was proved in \cite[Theorem 8.8]{OB}.

\begin{thm}[\cite{OB}]\label{thm:regions}
If an arrangement $\mA\subseteq\mA_m^n$ is transitive, then $\Phi_\mA$ is a bijection between the set $\Tmn(\mA)$ of trees and the set $\mR(\mA)$ of regions of $\mA$.
%(If $\mA$ is not transitive, then $\Phi$ is a surjection.)
\end{thm}

%%%%%%%%%%%%%%%%%%%%%%%%%%%%%%%%%%%%%%%%

\section{Main results}\label{sec:main}
In this section we state our main bijective result (Theorem~\ref{thm:main}), which is an explicit bijection between the set of faces of any ``strongly transitive'' arrangement and some set of marked labeled trees.

\begin{definition}
An arrangement $\mA\subseteq \mA_m^n$ is \emph{strongly transitive} if for all distinct indices $i,j,k\in[n]$ and integers $s,t\geq 0$ the following holds: \\
\centerline{if $\{x_i-x_j=s\}\notin \mA$ and $\{x_j-x_k=t\}\notin \mA$, then $\{x_i-x_k=s+t\}\notin \mA$.}
\end{definition}
Note that strongly transitive arrangements are transitive. The only difference between transitive and strongly transitive arrangements is that the conclusion $\{x_i-x_k=s+t\}\notin \mA$ needs to hold even in the cases ($s=0$ and $i<j$) or ($t=0$ and $j<k$).
In fact, an arrangement $\mA\subseteq \mA_m^n$ is strongly transitive if and only if $\pi(\mA)$ is transitive for every permutation $\pi\in \fS_n$ (where the action of $\pi$ on $\RR^n$ is the permutation of coordinates: $\pi\cdot (x_1,\ldots,x_n):=(x_{\pi^{-1}(1)},\ldots,x_{\pi^{-1}(n)})$).

\begin{example} 
The (extended) Catalan, Shi and semiorder arrangements are strongly transitive. 
The Linial arrangement is transitive, but not strongly transitive. 
Any transitive arrangement containing the braid arrangement (that is, containing all the hyperplanes of the form $\{x_i-x_j=0\}$) is strongly transitive. 
\end{example}

\begin{definition}
A \emph{marked $(m,n)$-tree} is a pair $(T,\mu)$, where $T\in\mT_m^n$ is a $(m,n)$-tree and $\mu$ is a set of cadet edges of $T$ such that if an edge $e\in \mu$ is of the form $e=\{j,0\cadet(j)\}$ then $j<0\cadet(j)$. We refer to the edges in $\mu$ as the \emph{marked edges}. 

We denote by $\oT_m^n$ the set of marked $(m,n)$-trees. 
\end{definition}

Note that the marked edges of a marked tree in $\oT_m^n$ form a collection of vertex-disjoint paths. A result of Levear \cite{Levear} is that the faces of the Catalan arrangement $\mA_m^n$ are in bijection with the set $\ov{\mT}_m^n$ of marked $(m,n)$-trees. In order to define a bijection for any strongly transitive arrangement we need some additional definitions.

\begin{definition}
Let $(T,\mu)\in\oT_m^n$ be a marked $(m,n)$-tree. For nodes $i,j\in[n]$, we write $i\musim j$ if $i=j$ or $i\neq j$ and all the edges on the path of $T$ between $i$ and $j$ are marked (and we write $i\nmusim j$ otherwise). This is an equivalence relation, and we call its equivalence classes the \emph{blocks} of $(T,\mu)$. Further, we denote by $\blocks(\mu)$ the set of blocks of $(T,\mu)$ (this is a partition of $[n]$).
\end{definition}

%$$\bPhi(T,\mu)=\left(\bigcap_{\substack{\{i,j\}\in \mu \\ i=s\child(j)}}\{x_i-x_j=s\}\right)\cap\left(\bigcap_{\substack{(i,j,s)\in \tr_m^n\\ i\nmusim j \\ i\preceqT s\child(j)}}\{x_i-x_j< s\}\right)\cap \left(\bigcap_{\substack{(i,j,s)\in \tr_m^n\\i\nmusim j \\ i\succT s\child(j)}}\{x_i-x_j>s\}\right).$$

\begin{definition}\label{def:A-connected-tree}
Let $\mA\subseteq \mA_m^n$ be an arrangement. 

A block $B\subseteq [n]$ of a marked tree $(T,\mu)\in \oT_m^n$ is called \emph{$\mA$-connected} if 
the graph $G$ with vertex set $B$ and edge set 
$$E=\big\{\,\{i,j\}\mid i,j\in B~\textrm{such that}~\{x_i-x_j=\driftT(i)-\driftT(j)\}\in \mA\,\big\}$$ 
is connected. We say that the marked tree $(T,\mu)$ is \emph{$\mA$-connected} if every block of $(T,\mu)$ is $\mA$-connected.

We say that a marked tree $(T,\mu)\in\oTmn$ satisfies the \emph{$\mA$-cadet condition} if every non-marked cadet edge $e=\{i,j\}$ of $(T,\mu)$, with $i=s\cadet(j)$, satisfies ($s=0$ and $i<j$) or (there exists $i'\musim i$ and $j'\musim j$ such that the hyperplane $\{x_{i'}-x_{j'}=\driftT(i')-\driftT(j')\}$ is in $\mA$).
 
We define $\oTmn(\mA)$ as the set of marked trees in $\oTmn$ which are $\mA$-connected and satisfy the $\mA$-cadet condition.
% 
%We denote by $\oTmn(\mA)$ the set of marked trees $(T,\mu)$ in $\oT_m^n$ such that 
%\bitem
%\item every block of $(T,\mu)$ is $\mA$-connected, 
%\item every non-marked cadet edge $e=\{i,j\}$ of $(T,\mu)$, with $i=s\cadet(j)$, satisfies ($s=0$ and $i<j$) or (there exists $i'\musim i$ and $j'\musim j$ such that the hyperplane $\{x_{i'}-x_{j'}=\driftT(i')-\driftT(j')\}$ is in $\mA$).
%\eitem
%% OB (Note to self): if $s=0$ and there exists no such hyperplane, then $i$ and $j$ are in fact alone in their blocks assuming that $\mA$ is strongly transitive. 
\end{definition}

\begin{definition}
Let $\mA\subseteq \mA_m^n$. We associate to each marked tree $(T,\mu)$ in $\oTmn(\mA)$ a polyhedron $\bPhi_\mA(M,\mu)\subseteq \RR^n$ defined as follows:
\begin{eqnarray*}
\bPhi_\mA(T,\mu)&=&\left(\bigcap_{\substack{\{i,j\}\in \mu \\ i=s\child(j)}}\{x_i-x_j=s\}\right)\\
&&\cap\left(\bigcap_{\substack{(i,j,s)\in \trmn(\mA)\\ i\nmusim j \\ i\precT s\child(j)}}\{x_i-x_j< s\}\right)\cap \left(\bigcap_{\substack{(i,j,s)\in \trmn(\mA)\\i\nmusim j \\ i\succeqT s\child(j)}}\{x_i-x_j>s\}\right).
\end{eqnarray*}
\end{definition}

\begin{rk}\label{rk:eqL}
It is easy to check that for any marked tree $(T,\mu)$ in $\oTmn(\mA)$, one has
\begin{equation}\label{eq:L2}
\bigcap_{\substack{\{i,j\}\in \mu \\ i=s\child(j)}} \{x_i-x_j=s\} 
= \bigcap_{\substack{(i,j,s)\in \trmn(\mA)\\ i\musim j}}\{x_i-x_j=s\}.
\end{equation}
Equation~\eqref{eq:L2} holds because every block of $(T,\mu)$ is $\mA$-connected, hence both sides of~\eqref{eq:L2} consists of the set of points $(x_1,\ldots,x_n)\in \RR^n$ such that $x_i-x_j=\driftT(i)-\driftT(j)$ for all $i\musim j$.

Consequently, the polyhedron $\bPhi_\mA(T,\mu)$ is a face of $\mA$ (of codimension $|\mu|$), unless it is empty.
\end{rk}

For an arrangement $\mA\subset \RR^n$, we denote by $\mF(\mA)$ the set of faces of $\mA$. We can now state our main result.

\begin{thm}\label{thm:main}
If an arrangement $\mA\subseteq \mA_m^n$ is strongly transitive, then $\bPhi_\mA$ is a bijection between the set $\oTmn(\mA)$ of marked trees and the set $\mF(\mA)$ of faces of $\mA$. The number of marked edges of $(T,\mu)$ is equal to the codimension of the corresponding face $\bPhi_\mA(T,\mu)$.
\end{thm}

Theorem~\ref{thm:main} is illustrated in Figure~\ref{fig:exp-face-bijection} for an arrangement in dimension 3.

\fig{width=\linewidth}{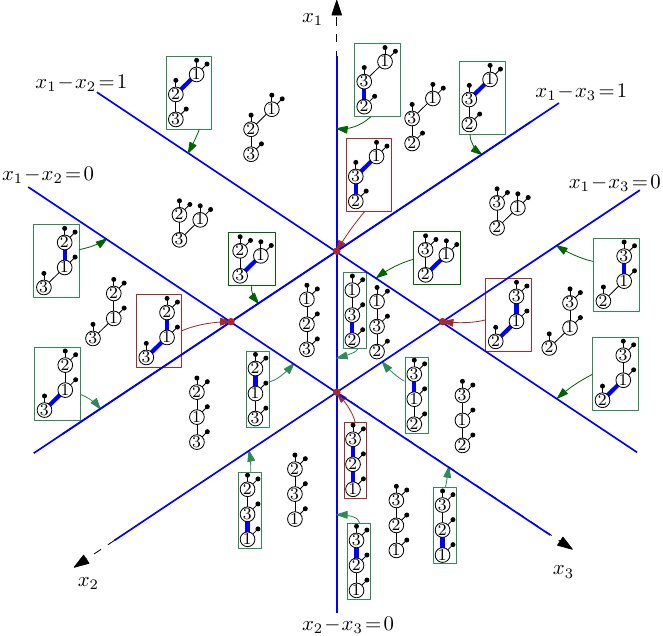}{The bijection $\bPhi_\mA$ for a strongly transitive arrangement $\mA\subseteq \mA_1^3$. The marked edges of the trees in $\oT_1^3(\mA)$ are indicated by bold lines.}

The special cases of Theorem~\ref{thm:main} corresponding to $\mA=\mA_m^n$ (the $m$-Catalan arrangement) and $\mA=\mA_{[-m+1;m]}^n$ (the $m$-Shi arrangement) give bijections which are the same as that of Levear \cite{Levear} (up to small changes in the presentation of the trees). We will discuss these special cases, and some others, in the following section. For now let us simply add a small remark about which arrangements of the form $\mA_S^n$ are strongly transitive. 

\begin{lemma}\label{lem:transitive-sets}
Let $S\subseteq \ZZ$ be a set of integers. The following are equivalent:
\bitem
\item[(i)] $\mA_S^n$ is strongly transitive for every integer $n> 0$,
\item[(ii)] $\mA_S^n$ is strongly transitive for at least one integer $n\geq 3$,
\item[(iii)] for all integers $s,t\notin S$, if $st\geq 0$ then $s+t\notin S$ and if $st\leq 0$ then $s-t,t-s\notin S$.
\eitem
\end{lemma}

\begin{proof}
Clearly, (i) implies (ii). Next, we prove that (ii) implies (iii). Suppose that $\mA_S^n$ is transitive for some $n\geq 3$. We consider some integers $s,t\notin S$, and distinguish four cases in order to prove (iii).
\bitem
\item[Case $s,t\geq 0$:] Since $\{x_1-x_2=s\}\notin \mA_S^n$ and $\{x_2-x_3=t\}\notin \mA_S^n$, we get $\{x_1-x_3=s+t\}\notin \mA_S^n$, hence $s+t\notin S$. 
\item[Case $s,t\leq 0$:] Since $\{x_3-x_2=-s\}\notin \mA_S^n$ and $\{x_2-x_1=-t\}\notin \mA_S^n$, we get $\{x_3-x_1=-s-t\}\notin \mA_S^n$, hence $s+t\notin S$. 
\item[Case $s\geq 0\geq t$:] Since $\{x_1-x_3=s\}\notin \mA_S^n$ and $\{x_3-x_2=-t\}\notin \mA_S^n$, we get $\{x_1-x_2=s-t\}\notin \mA_S^n$, hence $s-t\notin S$. Since $\{x_3-x_1=-t\}\notin \mA_S^n$ and $\{x_1-x_2=s\}\notin \mA_S^n$, we get $\{x_3-x_2=s-t\}\notin \mA_S^n$, hence $t-s\notin S$. 
\item[Case $s\leq 0\leq t$:] Since $\{x_1-x_3=t\}\notin \mA_S^n$ and $\{x_3-x_2=-s\}\notin \mA_S^n$, we get $\{x_1-x_2=t-s\}\notin \mA_S^n$, hence $t-s\notin S$. Since $\{x_3-x_1=-s\}\notin \mA_S^n$ and $\{x_1-x_2=t\}\notin \mA_S^n$, we get $\{x_3-x_2=t-s\}\notin \mA_S^n$, hence $s-t\notin S$. 
\eitem
It remains to prove that (iii) implies (i). Suppose that $S$ satisfies (iii), and fix $n>0$. We consider distinct integers $i,j,k\in [n]$, and $s,t\in \NN$ such that $\{x_i-x_j=s\}\notin \mA_S^n$ and $\{x_j-x_k=t\}\notin \mA_S^n$. We need to prove $\{x_i-x_k=s+t\}\notin \mA_S^n$, and distinguish four cases.
\bitem
\item[Case $i<j<k$:] Since $s,t\notin S$, we get $s+t\notin S$, hence $\{x_i-x_k=s+t\}\notin \mA_S^n$. 
\item[Case $i>j>k$:] Since $-s,-t\notin S$, we get $-s-t\notin S$, hence $\{x_i-x_k=s+t\}\notin \mA_S^n$. 
\item[Case $i<j>k$:] Since $s,-t\notin S$, we get $s+t,-t-s\notin S$, hence $\{x_i-x_k=s+t\}\notin \mA_S^n$. 
\item[Case $i>j<k$:] Since $-s,t\notin S$, we get $-s-t,t+s\notin S$, hence $\{x_i-x_k=s+t\}\notin \mA_S^n$. 
\eitem
\end{proof}
%%%%%%%%%%%%%%%%%%%%%%%%%%%%%%%%%%

%%%%%%%%%%%%%%%%%%%%%%%%%%%%%%%%%%%%%%%%%%%%%%%%%%%%%%%%%%%%%%%%%%%%%%%%%%%%%%%%%%

\section{Examples}\label{sec:exp}
In this section we apply Theorem~\ref{thm:main} to several families of braid-type arrangements. 

%%
%%Given a set of integers $S\subseteq \ZZ$ we define, for every dimension $n>0$, the arrangement $\mA_S^n\subseteq \mA_m^n\subset \RR^n$ as follows:
%%$$\mA_S^n:=\bigcup_{\substack{1\leq i<j\leq n\\ s\in S}}\{x_i-x_j=s\}.$$
%%The \emph{Catalan arrangement} is $\mA_{[-1;1]}^n=\mA_1^n$ and the \emph{Shi arrangement} is $\mA_{[0;1]}^n$. For $m>0$ the \emph{$m$-extended Catalan arrangement} is $\mA_m^n$ and the \emph{$m$-extended Shi arrangement} is $\mA_{[-m+1;m]}^n$.

\subsection{From the Catalan arrangement to the Shi arrangement, and back}
Theorem~\ref{thm:main} readily implies the following result of Levear \cite{Levear} about the $m$-Catalan arrangement $\mA_m^n$.
\begin{cor}[\cite{Levear}]\label{cor:Cat}
Let $m,n$ be positive integers. The faces (of codimension $k$) of the $m$-Catalan arrangement $\mA_m^n=\mA_{[-m;m]}^n$ are in bijection with the marked trees in $\oTmn$ (having $k$ marked edges).
\end{cor}

\begin{proof}
It is immediate that $\mA_m^n$ is strongly transitive.
We can therefore apply Theorem~\ref{thm:main} to $\mA_m^n$, and we only need to check that $\oTmn(\mA_m^n)=\oTmn$. Before starting this proof, let us observe that for any marked tree $(T,\mu)\in\oTmn$ and any nodes $i,j$ such that $i=s\cadet(j)$ one has 
$$\{x_i-x_j=\driftT(i)-\driftT(j)\}=\{x_i-x_j=s\}\in \mA_m^n.$$

We will now check that any tree in $\oTmn$ is $\mA_m^n$-connected. Let $B$ be a block of a marked tree $(T,\mu)\in\oTmn$, and let $G$ be the graph with vertex set $B$ and edge set 
$$E=\{\{i,j\}\mid i,j\in B~\textrm{ such that }~\{x_i-x_j=\driftT(i)-\driftT(j)\}\in \mA\}.$$ 
The block $B$ corresponds to a path of cadet edges in $T$, so $B$ is of the form $B=\{i_1,\ldots,i_\ell\}$, where for all $j<k$, $i_{j+1}=s_j\cadet(i_j)$ for some $s_j$. Hence, by the above observation, the graph $G$ contains the path $i_1,\ldots,i_\ell$. Thus $G$ is connected. 
This shows that any marked tree $(T,\mu)\in\oTmn$ is $\mA_m^n$-connected. 

It remains to show that any marked tree in $\oTmn$ satisfies the $\mA_m^n$-cadet condition. This is obvious since, by the above observation, every non-marked cadet edge $e=\{i,j\}$ of $(T,\mu)$ satisfies $\{x_i-x_j=\driftT(i)-\driftT(j)\}\in \mA_m^n$.
Thus $\oTmn(\mA_m^n)=\oTmn$, as claimed.
%We observe that for any marked tree $(T,\mu)\in\oTmn$ and any nodes $i,j$ such that $i=s\cadet(j)$ one has 
%$$\{x_i-x_j=\driftT(i)-\driftT(j)\}=\{x_i-x_j=s\}\in \mA_m^n.$$
%This clearly implies that any tree $(T,\mu)\in \oTmn$ is $\mA_m^n$-connected. Moreover, $(T,\mu)$ satisfies the $\mA$-cadet condition since for every non-marked cadet edge $e=\{i,j\}$ of $(T,\mu)$ the hyperplane $\{x_i-x_j=\driftT(i)-\driftT(j)\}$ is in $\mA_m^n$. Thus $(T,\mu)$ is in $\oTmn(\mA_m^n)$, and $\oTmn(\mA_m^n)=\oTmn$ as wanted.
\end{proof}

The other family of arrangements treated in \cite{Levear} consists of the \emph{$m$-Shi arrangements}
$$\mShi_m^n=\mA_{[-m+1;m]}^n=\bigcup_{\substack{1\leq i<j\leq n\\s\in[-m+1;m]}}\{x_i-x_j=s\}.$$
It is easy to check that the arrangement $\mShi_m^n$ is strongly transitive. Applying Theorem~\ref{thm:main} we recover the following result of \cite{Levear} about Shi arrangements.
\begin{cor}[\cite{Levear}]\label{cor:Shi}
Let $m,n$ be positive integers. The faces (of codimension $k$) of the $m$-Shi arrangement $\mShi_m^n$ are in bijection with the set of marked trees $(T,\mu)$ in $\oTmn$ (having $k$ marked edges) such that if $i=m\child(j)$ for some nodes $i,j\in [n]$ of $T$, then $i<j$.
\end{cor}

In order to prove of Corollary~\ref{cor:Shi}, it suffices to show that the set of marked trees described there corresponds precisely to the set $\oTmn(\mShi_m^n)$. We will actually prove a more general result below, which interpolates between the case of the Catalan arrangement and the case of the Shi arrangement.

Observe that the Shi arrangement $\mShi_m^n$ satisfies $\mA_{m-1}^n\subseteq \mShi_m^n\subseteq \mA_m^n$. It is easy to see that, for all $m>0$, any arrangement $\mA$ such that $\mA_{m-1}^n\subseteq \mA \subseteq\mA_m^n$ is strongly transitive. We now describe a family of arrangements $\mA$ which interpolates between $\mA_{m-1}^n$ and $\mA_{m}^n$ and for which the set $\oTmn(\mA)$ admits a simple description.

Let $$R_n:=\{(i,j)\in [n]^2\mid i\neq j\}~\textrm{ and }~R_n^+:=\{(i,j)\in [n]^2\mid i<j\}.$$ 
For a subset $I\subseteq R_n$ we define the arrangement
$$\mB_{m,I}^n=\mA_{m-1}^n\cup\left(\bigcup_{(i,j)\in I}\{x_i-x_j=m\}\right).$$

We say that $I\subseteq R_n$ is an \emph{ideal} if the following holds for all $(i,j),(i',j')\in R_n$:\\
\centerline{if $(i,j)$ is in $I$ and $i'\leq i$ and $j'\geq j$, then $(i',j')$ is in $I$.}
Note that $\emptyset$, $R_n$ and $R_n^+$ are ideals, and that $\mA_{m-1}^n=\mB_{m,\emptyset}^n$, $\mA_{m}^n=\mB_{m,R_n}^n$, and $\mShi_{m}^n=\mB_{m,R_n^+}^n$.

\begin{thm}\label{thm:Cat2Shi} 
Let $m,n$ be positive integers. For any ideal $I\subseteq R_n$, the faces (of codimension $k$) of the arrangement $\mA=\mB_{m,I}^n$ are in bijection, via the bijection $\bPhi_\mA$, with the set of marked trees $(T,\mu)\in \oTmn$ (with $k$ marked edges) such that if $i=m\child(j)$ for some nodes $i,j\in [n]$ of $T$, then $(i,j)$ is in $I$.
\end{thm}

\begin{example}
The arrangement in Figure~\ref{fig:exp-face-bijection} is $\mB_{1,I}^3$, where $I=\{(1,2),(1,3)\}$ (which is an ideal). The trees indexing the faces of this arrangement are those in $\oT_3^1$ for which the $i\neq 1\child(j)$ unless $(i,j)\in \{(1,2),(1,3)\}$. 
\end{example}

Note that Theorem~\ref{thm:Cat2Shi} generalizes both Corollary~\ref{cor:Cat} (which corresponds to $I=R_n$) and Corollary~\ref{cor:Shi} (which corresponds to $I=R_n^+$).

\begin{proof}
We apply Theorem~\ref{thm:main} to the strongly transitive arrangement $\mA=\mB_{m,I}^n$, and need to check that $\oTmn(\mA)$ is the set of marked trees $(T,\mu)$ in $\oTmn$ such that if $i=m\child(j)$ for some nodes $i,j\in [n]$ of $T$, then $(i,j)$ is in $I$. Before starting this proof, let us observe that for any marked tree $(T,\mu)\in\oTmn$ and any nodes $i,j$ such that $i=s\cadet(j)$ the hyperplane
$\{x_i-x_j=\driftT(i)-\driftT(j)\}=\{x_i-x_j=s\}$ is in $\mA$ unless $s=m$ and $(i,j)\notin I$. 

Now we will determine under which conditions a marked tree $(T,\mu)\in\oTmn$ is $\mA$-connected.
 Let $B$ be a block of $(T,\mu)$, and let $G$ be the graph with vertex set $B$ and edge set 
$$E=\{\{i,j\}\mid i,j\in B~\textrm{such that}~\{x_i-x_j=\driftT(i)-\driftT(j)\}\in \mA\}.$$
Recall that $B$ is of the form $B=\{i_1,\ldots,i_\ell\}$, where for all $k\in [\ell-1]$, $i_{k+1}=s_k\cadet(i_k)$ for some $s_k\leq m$. By the above observation, the edge $\{i_k,i_{k+1}\}$ is in $E$ whenever $s_k<m$. 
Hence the graph $G$ is connected if and only if for all $k\in[\ell-1]$ such that $i_{k+1}=m\cadet(i_k)$ there exist $k'\leq k$ and $k''\geq k+1$ such that $\driftT(i_{k'})=\driftT(i_k)$ and $\driftT(i_{k''})=\driftT(i_{k+1})$ (so that $\driftT(i_{k''})-\driftT(i_{k'})=m$) and $(i_{k''},i_{k'})\in I$ (so that $\{i_{k''},i_{k'}\}\in E$). 
%the edge $\{i_{j'},i_{j''}\}$ is in $E$. This holds if and only if $\driftT(i_{j'})=\driftT(i_j)$ and $\driftT(i_{j''})=\driftT(i_{j+1})$ (so that $\driftT(i_{j''})-\driftT(i_{j'})=m$) and $(i_{j'},i_{j''})\in I$ (so that $\{i_{j'},i_{j''}\}\in E$). 
Moreover, $\driftT(i_{k'})=\driftT(i_k)$ and $k'\leq k$ imply $i_{k'}\leq i_k$ (since $(T,\mu)\in \oTmn$) and similarly $\driftT(i_{k''})=\driftT(i_{k+1})$ and $k''\geq k+1$ imply $i_{k''}\geq i_{k+1}$. Hence, there exists $k'$ and $k''$ satisfying the above conditions if and only if $(i_{k+1},i_{k})\in I$ (because $I$ is an ideal). 
This shows that a marked tree $(T,\mu)\in\oTmn$ is $\mA$-connected if and only if $(i,j)$ is in~$I$ for every marked edge $\{i,j\}$ such that $i=m\child(j)$.
 
A similar reasoning shows that $(T,\mu)$ satisfies the $\mA$-cadet condition if and only if $(i,j)$ is in~$I$ for every non-marked edge $\{i,j\}$ such that $i=m\child(j)$.
\end{proof}

\begin{rk}
It is easy to see that there exist ideals 
$$\emptyset =I_0\subset I_1\subset I_2\cdots \subset I_{{n \choose 2}}=R_n^+ \subset \cdots\subset I_{n(n-1)}=R_n,$$
such that $|I_j\setminus I_{j-1}|=1$ for all $i\in [n(n-1)]$. Hence, upon denoting $\mB_p=\mB_{m,I_p}^n$ for all $p\in[0;n(n-1)]$, one gets a sequence of strongly transitive arrangements $\mB_0,\ldots,\mB_{n(n-1)}$ for which Theorem~\ref{thm:Cat2Shi} applies and such that 
$$\mB_{0}=\mA_{m-1}^n,\quad \quad \mB_{{n \choose 2}}=\mShi_m^n, \quad \quad \mB_{n(n-1)}=\mA_m^n,$$
and for all $p\in[n(n-1)]$, $\mB_{p}$ is obtained from $\mB_{p-1}$ by adding a single hyperplane.
\end{rk}

%In the next section we apply Theorem~\ref{thm:main} to additional families of hyperplane arrangements which generalize the Catalan and Shi arrangements.

\subsection{Multi-Catalan arrangements}
Let $n$ be a positive integer. Given a $n$-tuple of integers $\mm=(m_1,\ldots,m_n)\in \NN^n$, we define the \emph{$\mm$-Catalan arrangement} as
$$\mA_\mm:=\bigcup_{\substack{1\leq i<j\leq n\\ s\in [-m_i;m_j]}}\{x_i-x_j=s\}.$$

It is easy to see that the arrangement $\mA_\mm$ is strongly transitive for all $\mm\in \NN^n$.
We now want to describe the set of trees indexing its faces via the bijection $\Phi_{\mA_\mm}$. Let $m=\max(m_i\mid i\in[n])$.
Given a marked tree $(T,\mu)\in \oTmn$, we define the \emph{$\mm$-reach} of a node $j\in[n]$ of $T$ as
$$r_\mm(j):=\max(m_{k}+\driftT(k)-\driftT(j)\mid k\musim j\textrm{ and }k \textrm{ ancestor of }j),$$
where the node $j$ is considered an ancestor of itself. 

\begin{prop}
Let $n$ be a positive integer, let $\mm=(m_1,\ldots,m_n)\in \NN^n$, and let $m=\max(m_i\mid i\in[n])$. 
The faces of the $\mm$-Catalan arrangement $\mA_\mm$ are in bijection (via the bijection $\Phi_{\mA_\mm}$) with the set of marked trees $(T,\mu)\in \oTmn$ such that for every node $j\in[n]$, the vertex $s\child(j)$ is a leaf for all $s>r_\mm(j)$.
\end{prop}

\begin{proof}
We apply Theorem~\ref{thm:main} to the strongly transitive arrangement $\mA_\mm\subseteq \mA_m^n$. Let $\mT_\mm$ be the set of marked trees in $\oTmn$ such that for every node $j\in[n]$, the vertex $s\child(j)$ is a leaf for all $s>r_\mm(j)$.
We need to check that $\oTmn(\mA_\mm)=\mT_\mm$.  
Before starting this proof, let us observe that for any marked tree $(T,\mu)\in\oTmn$ and any nodes $i,j\in [n]$ such that $i$ is a descendant of $j$, the hyperplane 
$\{x_i-x_j=\driftT(i)-\driftT(j)\}$ is in $\mA_\mm$ if and only if $\driftT(i) \leq m_j+\driftT(j)$.

Now we will determine under which conditions a marked tree $(T,\mu)\in\oTmn$ is $\mA_\mm$-connected.
 Let $B$ be a block of $(T,\mu)$, and let $G$ be the graph with vertex set $B$ and edge set 
$$E=\{\{i,j\}\mid i,j\in B~\textrm{such that}~\{x_i-x_j=\driftT(i)-\driftT(j)\}\in \mA_\mm\}.$$
Recall that $B$ is of the form $B=\{i_1,\ldots,i_\ell\}$, where for all $j\in [\ell-1]$, $i_{j+1}=s_j\cadet(i_{j})$ for some $s_j\leq m$. Observe also that $G$ is connected if and only if for all for all $j\in [\ell-1]$, the vertex $i_{j+1}$ is connected (by a path) to a vertex $i_{k}$ for some $k\leq j$. By the above observation, it is easy to see that this happens if and only if $s_j\leq r_{\mm}(j)$. This shows that a marked tree $(T,\mu)\in\oTmn$ is $\mA_\mm$-connected if and only if $s\leq r_{\mm}(j)$ for every marked cadet edge $\{i,j\}$ such that $i=s\cadet(j)$.

A similar reasoning shows that a tree $(T,\mu)\in\oTmn$ satisfies the $\mA_\mm$-cadet condition if and only if $s\leq r_{\mm}(j)$ for every non-marked cadet edge $\{i,j\}$ such that $i=s\cadet(j)$.
%
%We will now check that any tree in $\oTmn$ is $\mA_m^n$-connected. Let $B$ be a block of a marked tree $(T,\mu)\in\oTmn$, and let $G$ be the graph with vertex set $B$ and edge set 
%$$E=\{\{i,j\}\mid i,j\in B,~\{x_i-x_j=\driftT(i)-\driftT(j)\}\in \mA\}.$$ 
%The block $B$ corresponds to a path of cadet edges in $T$, so $B$ is of the form $B=\{i_1,\ldots,i_\ell\}$, where for all $j<k$, $i_{j+1}=s_j\cadet(i_j)$ for some $s_j$. Hence by the above observation the graph $G$ contain path $i_1,\ldots,i_\ell$, hence it is connected. 
%This shows that any marked tree $(T,\mu)\in\oTmn$ is $\mA_m^n$-connected. 
\end{proof}

%%%%%%%%%%%%%%%%%%%%%%%%%%%%%%%%%%%%%%%%%%%%%%%%%%%%%%%%%%%%%%%%%%%%%%%%%%%%%%%%%%

\section{Generating function for faces of symmetric transitive arrangements}\label{sec:GF}
In this section we focus on \emph{symmetric} braid-type arrangements (see definition below). Examples of strongly transitive symmetric  braid-type arrangements include the $m$-Catalan arrangement $\mA_m^n$, and the \emph{$m$-semiorder arrangement} $\mA_{[-m;m]\setminus\{0\}}^n$. We first treat these two families of arrangements and characterize the generating functions of their faces. We then turn to the general situation, and give a general method for characterizing the generating function of faces for any strongly transitive symmetric arrangement. 

An arrangement $\mA\subset \RR^n$ is \emph{symmetric} if $\pi(\mA)=\mA$ for all $\pi\in \fS_n$. It is easy to see that the symmetric braid-type arrangements are the arrangements of the form $\mA_S^n$ for a set $S\subseteq \ZZ$ such that $S=-S$ (where $-S:=\{-s\mid s\in S\}$).
In view of Lemma~\ref{lem:transitive-sets}, one gets the following.
\begin{lemma} \label{lem:condition-sym-trans}
The strongly transitive symmetric arrangements in dimension $n\geq 3$ are precisely the arrangements of the form $\mA_S^n$, where the set $S\subseteq\ZZ $ satisfies 
\begin{equation}\label{eq:condition-sym-trans}
S=-S\textrm{ and }~ \forall s,t\in \NN\setminus S, \quad s+t\notin S.
\end{equation}
\end{lemma}

Given a finite set $S\subseteq \NN$, we define the \emph{face generating function} of the arrangements $\mA_S^n$ as 
$$F_{S}(x,u):=\sum_{n=0}^\infty \sum_{k=0}^{n}c_{n,k}\, u^k \frac{x^n}{n!},$$
where $c_{n,k}$ is the number of faces of dimension $k$ of $\mA_{S}^n$. %Note that the faces of any braid-type arrangement (in positive dimension) have dimension at least 1. 

\begin{rk}
The generating function of regions of the arrangements $\mA_S^n$ is $R_S(x):=\sum_{n=0}^\infty|\mR(\mA_S^n)|\frac{x^n}{n!}$. It corresponds to the specialization of $F_S(ux,1/u)$ at $u=0$. An equation characterizing $R_S(x)$ was given in \cite[Proposition 6.6]{OB} for any set $S$ such that the arrangements $\mA_S^n$ are transitive (see also \cite[Eq. (6.11)]{OB} for the case of a transitive set $S$ satisfying $S=-S$).
\end{rk}

Now we suppose that the set $S\subset \ZZ$ satisfies~\eqref{eq:condition-sym-trans}, so that $\mA_S^n$ is symmetric and strongly transitive for all $n\geq 0$. Let $m=\max(S)$, and let
$$\oT(S):= \bigcup_{n=0}^\infty \oTmn(\mA_S^n).$$
Applying Theorem~\ref{thm:main} gives the following result.
\begin{cor}\label{cor:thmGF}
Let $S\subseteq \ZZ$ be a finite set satisfying~\eqref{eq:condition-sym-trans}. The face generating function $F_S$ is given by
$$F_S(x,u)=\sum_{(T,\mu)\in \oT(S)}u^{\#\blocks(\mu)}\frac{x^{|T|}}{|T|!},$$
where $|T|$ is the number of nodes of the tree $T$, and $\#\blocks(\mu)$ is the number of blocks of the partition $\blocks(\mu)$.
\end{cor}

For a marked tree $(T,\mu)\in\oT(S)$ with at least one node, we call \emph{root-block} the block of the partition $\blocks(\mu)$ containing the root vertex. Below we will characterize the generating function $F_S(x,u)$ by considering the recursive decomposition of trees in $\oT(S)$ ``at the root-block''.

%%%%%%%%%%%%%%%%%%%%%%%%%%%%%%%%%%%%%%%%
%The \emph{Catalan arrangement} is $\mA_{[-1;1]}^n=\mA_1^n$, the \emph{Shi arrangement} is $\mA_{[0;1]}^n$, and the \emph{semiorder arrangement} is $\mA_{\{-1,1\}}^n$ (in dimension $n$). For $m>0$ the \emph{$m$-Catalan arrangement} is $\mA_m^n$, the \emph{$m$-Shi arrangement} is $\mA_{[-m+1;m]}^n$, and the \emph{$m$-semiorder arrangement} is $\mA_{[-m;m]\setminus\{0\}}^n$.

\subsection{Generating functions of Catalan and semiorder arrangements}
For the $m$-Catalan arrangement, Theorem~\ref{thm:main} states that the faces of $\mA_m^n$ of dimension $k$ are in bijection with the marked trees in $\oTmn$ having $k$ blocks. This result was first established by Levear \cite{Levear}, and allows one to recover a counting result due to Athanasiadis \cite{Athanasiadis:finite-field-method}. 
Below we show how to obtain the face generating function of the $m$-Catalan arrangement. 
%Let$$G_{[m;m]}(x,u)=\sum_{n=0}^\infty c_{n,k} t^k \frac{x^n}{n!},$$where $c_{n,k}$ is the number of faces of codimension $k$ of $\mA_m^n$.
\begin{prop}[\cite{Levear}]\label{prop:GFCatalan}
The generating function $G\equiv F_{[-m;m]}(x,u)$ counting the faces of the $m$-Catalan arrangements is characterized by the following equation: 
$$G=1+u\,\Om(\XX ,F),~\textrm{ where }~\Om(X,Y)=\frac{XY^{m+1}}{1-XY\frac{1-Y^m}{1-Y}}.$$
%\frac{x/(1-x) y^{m+1}}{1-(x/(1-x)) \frac{1-y^{m+1}}{1-y}}.$$
This gives the equivalent equation:
%$$G=1+(\XX )\left(\left(1+u\right)G^{m+1}-G\right).$$
$$G=e^{-x}+(1+u)(1-e^{-x})G^{m+1}.$$
\end{prop}
%In the case $m=1$, the above equation for $G=F_{[-1;1]}(x,u)$ is 
%$$G=1+\frac{(\XX )G^2}{t(1-(\XX )G)}.$$
Here and below, $e^{x}$ has to be understood as the formal power series $\sum_{n=0}^\infty \frac{x^n}{n!}\in\QQ[[x]]$.
\begin{proof}
Let $\oT_m:=\oT([-m;m])$. 
We consider the decomposition of the trees in $\oT_m$ at the root-block depicted in Figure~\ref{fig:root-block-decomposition}(a).
% The root-block can be any sequence of vertices $v_1,\ldots,v_k$ such that $v_1$ is the root vertex and for all $i<k$, $v_{i+1}=s_i\cadet(v_i)$ for some $s_i\in [0;m]$ with the condition $v_i<v_{i+1}$ if $s_i=0$.
The root-block of a tree $(T,\mu)\in\oT_m$ with at least one node
can be any sequence of nodes 
$$v_{1,1},v_{1,2}\ldots,v_{1,k_1},\, v_{2,1},v_{2,2},\ldots,v_{2,k_2},\,\ldots\,,v_{r,1},v_{r,2},\ldots,v_{r,k_r},$$
such that $r\geq 1$, $k_1,\ldots,k_r\geq 1$, $v_{1,1}$ is the root, and 
\bitem 
\item for all $i\in[r]$ and $j\in [k_i-1]$, $v_{i,j+1}=0\cadet(v_{i,j})$ and $v_{i,j+1}>v_{i,j}$,
\item for all $i\in [r-1]$, $v_{i+1,1}=s_i\cadet(v_{i,k_i})$ for some $s_i\in[m]$.
\eitem
\fig{width=\linewidth}{root-block-decomposition}{(a) Decomposition of a marked tree in $\oT_m=\oT([-m;m])$ at the root-block. (b) A partition of the set $\oT_m'=\oT([-m;m]\setminus\{0\})$.}
In this case, there are $\sum_{i=1}^{k-1}s_i$ subtrees rooted at nodes which are left siblings of the vertices in the root-block. There are also $m+1$ subtrees rooted at $v_{r,k_r}$, as represented in Figure~\ref{fig:root-block-decomposition}(a). This decomposition gives the equation
$$G=1+u\frac{(e^x-1)}{1-(e^x-1)\sum_{s=1}^m G^s}G^{m+1},$$
where the summand 1 corresponds to the tree with no node, the factor $u$ counts the root-block, the term $G^{m+1}$ corresponds to the subtrees attached to the node $v_{r,k_r}$, the terms $e^x-1$ correspond to the choice of the labels of vertices within each ``run'' $v_{i,1},v_{i,2},\ldots,v_{i,k_i}$, and the term $\sum_{s=1}^m G^s$ corresponds to the subtrees rooted at nodes which are left siblings of the vertices in the root-block (with choice of the parameters $s_1,\ldots,s_{k-1}$). 
\end{proof}

We now turn our attention to the ($m$-extended) semiorder arrangement $\mA:=\mA^n_{[m;m]\setminus \{0\}}$. The arrangement $\mA$ is strongly transitive.
The bijection $\Phi_\mA$ for this arrangement maps the faces of $\mA$ with the marked trees $(T,\mu)\in \mT_m^n$ such that 
\bitem 
\item if a block of $\blocks(\mu)$ is made of a sequence of nodes $i_1,\ldots,i_\ell$ with $\ell>1$ and $i_{k+1}=s_k\cadet(i_k)$ for all $k\in [\ell-1]$, then there is $k\in[\ell-1]$ such that $s_k>0$,
\item and if $i=0\cadet(j)$ and both $i$ and $j$ are alone in their respective blocks, then $i<j$.
\eitem
The bijection $\phi_A$ is illustrated in Figure~\ref{fig:exp-face-bijection-semiorder}.

\fig{width=\linewidth}{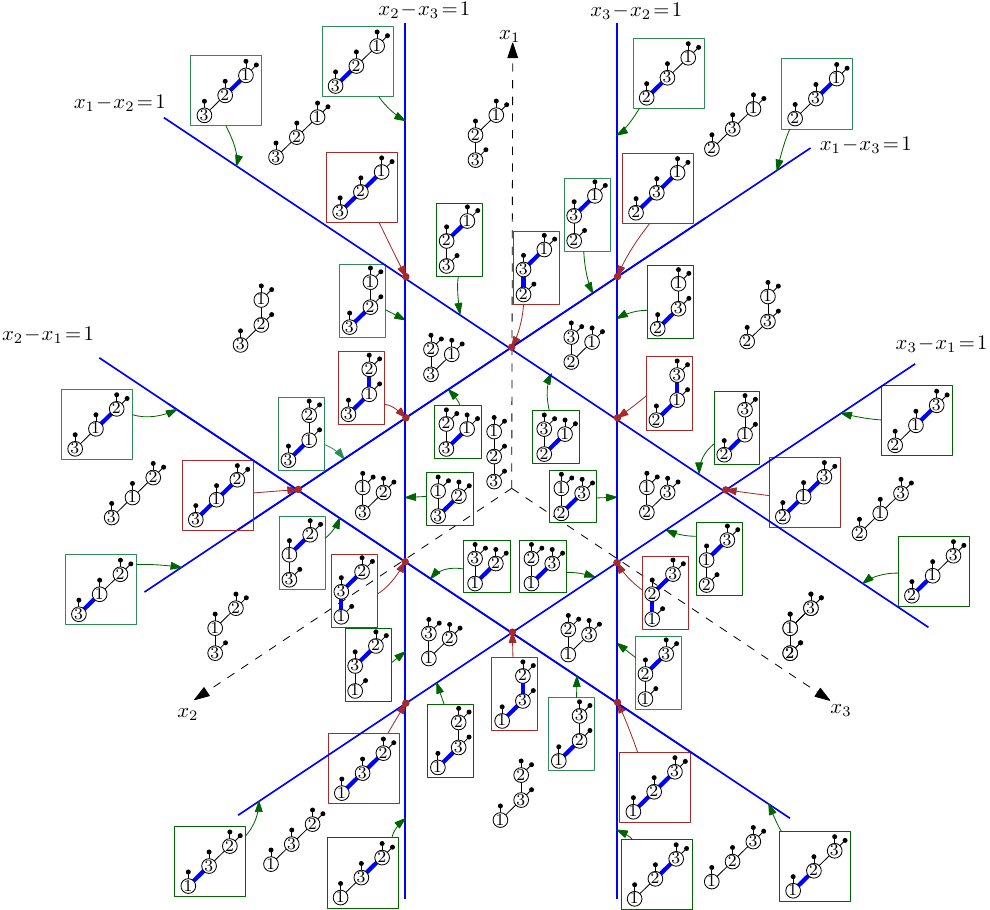}{The bijection $\bPhi_\mA$ for the semiorder arrangement $\mA\subseteq \mA_{\{-1,1\}}^3$.}

We will now characterize the face generating function of semiorder arrangements.
\begin{prop}\label{prop:GFsemiorder}
The generating function $H\equiv F_{[-m;m]\setminus \{0\}}(x,u)$ counting the faces of the $m$-extended semiorder arrangements is characterized by the following equation: 
$$H=1+(1-e^{-xu})H^{m+1}+u\,\wt\Om(\XX ,H),
~\textrm{ where }~\wt\Om(X,Y)=X^2\frac{Y^{m+2}-Y^{2m+2}}{1-Y-XY+XY^{m+1}}.$$
\end{prop}
%In the case $m=1$, the above equation for $H=F_{\{-1,1\}}(x,u)$ is 
%$$H=1+(1-e^{-x})H^2+\frac{(\XX )^2H^3}{t(1-(\XX )H)}.$$

\begin{proof}
Let $\oT_m':=\oT([-m;m]\setminus\{0\})$. 
We consider the decomposition of the trees in $\oT_m'$ at the root-block. The blocks of a tree in $(T,\mu)\in\oT_m'$ can be either reduced to a single node or made of an arbitrary sequence of vertices 
$$v_{1,1},v_{1,2}\ldots,v_{1,k_1},\, v_{2,1},v_{2,2},\ldots,v_{2,k_2},\,\ldots\,,v_{r,1},v_{r,2},\ldots,v_{r,k_r},$$
such that $r\geq 2$, $k_1,\ldots,k_r\geq 1$, $v_{1,1}$ is the root, and 
\bitem 
\item for all $i\in[r]$ and $j\in [k_i-1]$, $v_{i,j+1}=0\cadet(v_{i,j})$ and $v_{i,j+1}>v_{i,j}$,
\item for all $i\in [r-1]$, $v_{i+1,1}=s_i\cadet(v_{i,k_i})$ for some $s_i\in[m]$.
\eitem
Let $\oT_m''$ be the set of trees in $\oT_m'$ such that the root-block has several nodes, and let $\wt H=\sum_{(T,\mu)\in\oT_m''}u^{\#\blocks(\mu)}\frac{x^{|T|}}{|T|!}$.
 By reasoning as in the proof of Proposition~\ref{prop:GFCatalan}, one gets 
\begin{equation}\label{eq:eqHtilde}
\wt H=u\left(\frac{1}{1-(e^x-1)\sum_{s=1}^m H^s}-1\right)(e^x-1)H^{m+1}.
\end{equation}
Now, given a tree $(T,\mu)$ in $\oT_m'$ having at least one node we consider the maximal sequence of vertices $v_1,\ldots,v_k$ starting with the root $v_1$, such that $v_{i+1}=0\cadet(v_i)$ for all $i\in[k-1]$, and $v_1,\ldots,v_k$ are in different blocks (so $v_1,\ldots,v_{k-1}$ are alone in their respective blocks). This is represented in Figure~\ref{fig:root-block-decomposition}(b).

The set $\mH_1\subset \oT_m'$ of trees such that $v_k$ is not alone in its block contributes $H_1:=e^{xu}\wt H$ to the generating function $H$ (where the term $e^{xu}$ accounts for the vertices $v_1,\ldots,v_{k-1}$ whose labels must be decreasing). The set $\mH_2\subset \oT_m'$ of trees such that $v_k$ is alone in its block contributes $H_2:=(e^{xu}-1)(1+H(H^{m}-1))$, where the factor $(e^{xu}-1)$ accounts for the vertices $v_1,\ldots,v_k$, the term 1 accounts for the case where the children of $v_k$ are all leaves, while the term $H(H^{m}-1)$ corresponds to the case where $s\child(v_k)$ is a node for some $s>0$. 

This gives $H=1+H_1+H_2=e^{xu}+e^{xu}\wt H+(e^{xu}-1)H(H^{m}-1)$, hence
$$e^{xu}H=e^{xu}(1+\wt H)+(e^{xu}-1)H^{m+1}.$$ 
This together with~\eqref{eq:eqHtilde} prove the result. 
\end{proof}

\subsection{Generating functions of arbitrary symmetric transitive arrangements}

\begin{thm}\label{thm:GF}
Let $S\subset \ZZ$ be a finite set satisfying~\eqref{eq:condition-sym-trans}.
The face generating function $F_S(x,u)$ is characterized by a finite equation, which is computable from $S$. Precisely, if $0\in S$, then $F_S$ is \emph{determined by} (that is, is the unique formal power series solution of) a polynomial equation of the form $P(F_S(x,u),e^{x},u)=0$ for some trivariate polynomial $P$ with coefficients in $\QQ$ (hence $F_S(x,u)=\hat F(\XX,u)$ with $\hat F$ an algebraic series). If $0\notin S$ then $F_S$ is determined by a polynomial equation of the form $P(F_S(x,u),e^{xu},e^x,u)=0$ for some polynomial $P$ with coefficients in $\QQ$.
\end{thm}

%\OB{Add the result about the GF of faces of dimension 1.}
Propositions~\ref{prop:GFCatalan} and~\ref{prop:GFsemiorder} are special cases of Theorem~\ref{thm:GF}.
Our proof of Theorem~\ref{thm:GF} is constructive: we will explain how to derive the equation for $F_S\equiv F_S(x,u)$ for any set~$S$ satisfying~\eqref{eq:condition-sym-trans}.

For the rest of this section we fix a finite set $S\subset \ZZ$ satisfying~\eqref{eq:condition-sym-trans}, and we let $m=\max(S)$. We will establish Theorem~\ref{thm:GF} by describing the recursive decomposition of the trees in $\oT(S)$ at the root-block.
For a start, let us describe which marked trees in $\oT_m:=\bigcup_{n=0}^\infty \oTmn$ are in~$\oT(S)$. 

For a marked tree $(T,\mu)\in \oT_m$, we call \emph{shadow} of a block $B\in\blocks(\mu)$, the set 
$$\shadow(B):=\{\driftT(j)-\driftT(a)\mid j\in B\},$$
where $a\in B$ is the ancestor of all the nodes in $B$.

\begin{definition}
A \emph{$S$-shadow} is a finite set $D\subseteq \NN$ such that $0\in D$ and the graph $G_D$ having vertex set $D$ and edge set
$$E_D:=\{\{i,j\}\mid i,j\in D, ~i-j\in S \}$$
is connected. We let $\mD_S$ be the set of $S$-shadows. 
\end{definition}

The following lemma shows the relevance of $S$-shadows for counting trees in $\oT(S)$.
\begin{lemma}\label{lem:shadow}
\bitem
\item If $0\in S$, a marked tree $(T,\mu)\in \oTmn$ is $\mA_S^n$-connected if and only if every block $B$ of the partition $\blocks(\mu)$ satisfies $\shadow(B)\in \mD_S$.
\item If $0\notin S$, a marked tree $(T,\mu)\in \oTmn$ is $\mA_S^n$-connected if and only if every block $B$ of the partition $\blocks(\mu)$ satisfies $\shadow(B)\in \mD_S$ and ($|\shadow(B)|>1$ or $|B|=1$).
\eitem
\end{lemma}
Lemma~\ref{lem:shadow} follows easily from the definitions, and we leave its proof to the reader. Next, we focus on the $\mA_S^n$-cadet condition.
For an $S$-shadow $D\subset \NN$ we define
$$\start(D):=D\cap [m-1],~\textrm{ and } ~\eend(D):=(\max(D)-D)\cap [m-1],$$
where $\max(D)-D:=\{\max(D)-d\mid d\in D\}$.
For a block $B$ of the partition $\blocks(\mu)$ associated to a tree $(T,\mu)\in \oT(S)$, we use the notation $\start(B):=\start(\shadow(B))$ and $\eend(B):=\eend(\shadow(B))$.
We also define for $V,W\subset [m-1]$ the following \emph{connection set}
$$C_S^{V,W}=\{0\}\cup \{s\in [m]\mid \exists d\in \{0\}\cup V,~\exists d'\in \{0\}\cup W,\textrm{ such that }s+d+d'\in S\}.$$

The following lemma shows the relevance of connection sets for counting trees in $\mT(S)$.
\begin{lemma}\label{lem:connection}
 Let $(T,\mu)$ be a marked tree in $\oTmn$ which is $\mA_S^n$-connected.
\bitem
\item If $0\in S$, the tree $(T,\mu)$ belongs to $\oT(S)$ if and only if for every non-marked cadet edge $\{i,j\}$, one has $i=s\cadet(j)$ for some $s$ in $\mC_S^{\eend(B),\start(B')}$, where $B$ and $B'$ are the blocks of $\blocks(\mu)$ containing $j$ and $i$ respectively.
\item If $0\notin S$, the tree $(T,\mu)$ belongs to $\oT(S)$ if and only if for every non-marked cadet edge $\{i,j\}$, one has $i=s\cadet(j)$ for some $s$ in $\mC_S^{\eend(B),\start(B')}$, where $B$ and $B'$ are the blocks of $\blocks(\mu)$ containing $j$ and $i$ respectively and ($s>0$ or $|B|>1$ or $|B'|>1$ or $i<j$).
\eitem
\end{lemma}
\begin{proof} 
In the case $0\in S$, Lemma~\ref{lem:connection} follows easily from the definitions. Indeed, in this case 
$$C_S^{\eend(B),\start(B')}=\{s\in [0;m]\mid \exists d\in \{0\}\cup \eend(B),~\exists d'\in \{0\}\cup \start(B'),\textrm{ such that }s+d+d'\in S\},$$
and the condition $s\in\mC_S^{\eend(B),\start(B')}$ exactly translates the fact that the $\mA_S^n$-cadet condition holds for the cadet edge $\{i,j\}$. 

We now examine the case $0\notin S$. The only subtlety concern the situation $i=0\cadet(j)$ with $i>j$. In this case, the $\mA_S^n$-cadet condition holds for $\{i,j\}$ if and only if
\begin{equation}\label{eq:cond-case-s0}
\exists d\in \{0\}\cup \eend(B),~\exists d'\in \{0\}\cup \start(B')\textrm{ such that }d+d'\in S.
\end{equation} 
Hence we need to prove that Condition~\eqref{eq:cond-case-s0} holds if and only if $|B|>1$ or $|B'|>1$. If $|B|=|B'|=1$, then clearly~\eqref{eq:cond-case-s0} does not hold. If $|B|>1$, then by Lemma~\ref{lem:shadow}, $|\shadow(B)|>1$ and $\shadow(B)\in \mD_S$ (since $(T,\mu)$ is $\mA_S^n$-connected). In this case $\exists d\in\eend(B)\cap S$, hence~\eqref{eq:cond-case-s0} holds. Similarly, if $|B'|>1$, then $\exists d'\in\start(B')\cap S$, hence~\eqref{eq:cond-case-s0} holds.
\end{proof}

We will now use Lemmas~\ref{lem:shadow} and~\ref{lem:connection} to derive an equation for $F_S$.
For a marked tree $(T,\mu)\in\oT(S)$ with at least one node, we call \emph{root-shadow} the shadow of its root-block.
%For a marked tree $(T,\mu)\in\oT(S)$ with at least one node, we call \\emph{root-block} the block $B$ of the partition $\blocks(\mu)$ containing the root vertex, and we call $\shadow(B)$ the \emph{root-shadow} of $(T,\mu)$. 
For $U\subseteq [m-1]$ we define $\oT^U$ to be the set of marked trees on $\oT(S)$ (with at least one node) having root-shadow $D$ satisfying $\start(D)=U$. We let $F_S^U(x,u)$ be the generating function of those marked trees:
$$F^U_S(x,u):=\sum_{(T,\mu)\in\oT^U}u^{\#\blocks(\mu)}\frac{x^{|T|}}{|T|!}.$$ 
Clearly,
\begin{equation}\label{eq:F-FU}
F_S(x,u)=1+\sum_{U\subseteq[m-1]}F^U_S(x,u).
\end{equation}

Next we derive a system of equations for the generating functions $F_S^U$. For this we need to consider some generating functions of $S$-shadows, and of connection sets. For $U,V\in[m-1]$ we define $\mD_S^{U,V}$ as the set of $S$-shadows $D$ such that $\start(D)=U$ and $\eend(D)=V$. We define 
$$D_S^{U,V}(X,Y):=\sum_{D\in \mD_S^{U,V}}X^{|D|}Y^{\max(D)}\quad \textrm{ and }\quad C_S^{U,V}(Y):=\sum_{s\in \mC_S^{U,V}}Y^s,$$
where $|D|$ is the cardinality of $D$.

%, and $$C_S^{U,V}(Y):=\sum_{s\in \mC_S^{U,V}}Y^s.$$
%Note that $C_S^{U,V}(Y)$ is a polynomial in $Y$. 

\begin{example}
In the case $S=[-m;m]$ one has $C_S^{U,V}(Y)=1+Y+Y^2+\ldots+Y^m$ for all $U,V\subseteq [m-1]$, and $\sum_{U,V\subseteq [m-1]}D_S^{U,V}(X,Y)=\frac{X}{1-X(Y+Y^2+\ldots,Y^m)}$.
%$D_S^{\emptyset,\emptyset}=X+X^2Y^m\left(1+\frac{XY^m}{1-X(Y+Y^2+\ldots+Y^m)}\right)$. 
\end{example}

Using Lemmas~\ref{lem:shadow} and~\ref{lem:connection} we will establish the following result. %\OB{At the very least this requires a picture.}
\begin{cor} \label{cor:system-FSU}
Consider the series $F_S\equiv F_S(x,u)$ and $F_S^U\equiv F_S^U(x,u)$ for $U\subseteq [m-1]$.\\
If $0\in S$ or $U\neq \emptyset$, then
\begin{equation}\label{eq:system-FU}
F_S^U=\sum_{V\subseteq[m-1]}u\,D_S^{U,V}(\XX ,F_S)\cdot\bigg(1+\sum_{W\subseteq [m-1]}C_S^{V,W}(F_S)\cdot F_S^W\bigg).
\end{equation}
If $0\notin S$, then 
\begin{eqnarray}\label{eq:system-Fempty}
F_S^\emptyset&=&(1-e^{-xu})\bigg(1+\sum_{V\subseteq [m-1]}C_S^{\emptyset,V}(F_S)\cdot F_S^V\bigg) \nonumber\\
&&+\sum_{V\subseteq[m-1]}u\,\big(D_S^{\emptyset,V}(\XX ,F_S)-\one_{V=\emptyset}(\XX )\big)\cdot\bigg(1+\sum_{W\subseteq [m-1]}C_S^{V,W}(F_S)\cdot F_S^W\bigg),
\end{eqnarray}
where $\one_{V=\emptyset}$ is equal to 1 if $V=\emptyset$, and 0 otherwise.
\end{cor}
\begin{proof}
Let us start with the case $0\in S$. Let $\oT^{U,V}$ be the set of marked trees $(T,\mu)$ in $\oT^U$ having root-shadow in $\mD_S^{U,V}$, and let 
$$F_S^{U,V}(x,u):=\sum_{(T,\mu)\in\oT^{U,V}}u^{\#\blocks(\mu)}\frac{x^{|T|}}{|T|!}.$$ 
Clearly $F_S^U=\sum_{V\subseteq [m-1]}F_S^{U,V}$, and we want to show 
\begin{equation}\label{eq:proof-FSU}
F_S^{U,V}=u\,D_S^{U,V}(\XX ,F_S)\cdot\bigg(1+\sum_{W\subseteq [m-1]}C_S^{V,W}(F_S)\cdot F_S^W\bigg).
\end{equation}
For a marked tree $(T,\mu)$ in $\oT^{U,V}$ we call \emph{cap} the node of the root-block which is the descendant of all the other. Given a tree in $(T,\mu)$ we consider two parts:
\begin{compactitem}
\item[(a)] the root-block together with the subtrees attached to the non-cap vertices of the root-block,
\item[(b)] the subtrees attached to the cap.
\end{compactitem}
By Lemma~\ref{lem:shadow} the contribution of part (a) to $F_S^{U,V}$ is 
$$\sum_{D\in\mD_S^{U,V}}u(\XX )^{|D|}F_S^{\max(D)}=u\,D_S^{U,V}(\XX ,F_S),$$
where $D$ represent the root-shadow, the terms $(e^x-1)^{|D|}$ accounts for the labeling of the vertices in the root-block, and the term $F_S^{\max(D)}$ accounts for the subtrees attached to the non-cap vertices in the root-block (there are $\max(D)$ such subtrees).

By Lemma~\ref{lem:connection}, the contribution of part (b) to $F_S^{U,V}$ is 
$$1+\sum_{W\subseteq [m-1]}C_S^{V,W}(F_S)\cdot F_S^W,$$
where the term 1 corresponds to the case where the children of the cap are all leaves, while the term $C_S^{V,W}(F_S)\cdot F_S^W$ corresponds to the case where the cadet of the cap is in a block whose shadow $D'$ satisfies $\start(D')=W$ (and $C_S^{V,W}(F_S)$ accounts for the subtrees attached to the left siblings of the cadet of the cap). This gives~\eqref{eq:proof-FSU} and completes the proof of the case $0\in S$.

\newcommand{\oTe}{\wt{\mT}^\emptyset}

In the case $0\notin S$ one can prove~\eqref{eq:system-FU} for all $U\neq \emptyset$ as before. It remains to prove~\eqref{eq:system-Fempty}. One of the differences with the case $0\in S$, is that when $0\notin S$ the blocks with shadow $D=\{0\}$ contain only one node. Let $\oTe$ be the set of marked tree in $\oT^\emptyset$ whose root-block has more than one node, and let  $\wt F_S^\emptyset$ be the corresponding generating function.
% of marked trees in $\oTe_2$.% such that the root-block contains more than one node.
By reasoning as above, one gets
\begin{equation}\label{eq:tildeFempty}
\wt F_S^\emptyset=\sum_{V\subseteq[m-1]}u\,\big(D_S^{\emptyset,V}(\XX ,F_S)-\one_{V=\emptyset}(\XX )\big)\cdot\bigg(1+\sum_{W\subseteq [m-1]}C_S^{V,W}(F_S)\cdot F_S^W\bigg),
\end{equation}
where the term $-\one_{V=\emptyset}(\XX )$ accounts for the exclusion of the case $D=\{0\}$.

Given a tree $(T,\mu)$ in $\oT^\emptyset\setminus \oTe$, we consider the maximal sequence of nodes $v_1,\ldots,v_k$ starting with the root $v_1$, such that $v_{i+1}=0\cadet(v_i)$ for all $i\in[k-1]$, and $v_1,\ldots,v_k$ are alone in their respective blocks. Note that this implies  $v_1>v_2>\cdots>v_{k}$. The contribution to $F_S^\emptyset$ of the trees such that the children of $v_k$ are all leaves is $e^{xu}-1$. The contribution to $F_S^\emptyset$ of the trees such that $v_k$ has a $0\cadet$ $w$ in a block $B$ such that $\start(B)=\emptyset$ is $(e^{xu}-1)\wt F_S^\emptyset$ (since $w$ is not alone in its block by definition of $v_1,\ldots,v_k$). The  contribution to $F_S^\emptyset$ of the trees such that $v_k$ has an $s$-cadet $w$ in a block $B$ such that $(s,\start(B))\neq (0,\emptyset)$ is  
$$(e^{xu}-1)\bigg(\sum_{V\subseteq [m-1]}C_S^{\emptyset,V}(F_S)F_S^V-F_S^\emptyset\bigg).$$
Combining these results gives
$$ F_S^\emptyset-\wt F_S^\emptyset= (e^{xu}-1)\bigg(1+\wt F_S^\emptyset+\sum_{V\subseteq [m-1]}C_S^{\emptyset,V}(F_S)F_S^V-F_S^\emptyset\bigg).$$
Solving for $F_S^\emptyset$, one gets $\ds F_S^\emptyset=\big(1-e^{-xu}\big)\bigg(1+\sum_{V\subseteq [m-1]}C_S^{\emptyset,V}(F_S)F_S^V\bigg)+\wt F_S^\emptyset$. This, together with~\eqref{eq:tildeFempty}, gives~\eqref{eq:system-Fempty}.
\end{proof}

\begin{example}
For $S=[-m;m]$, one has $C_S^{U,V}(Y)=\frac{Y^{m+1}-1}{Y-1}$ for all $U,V\subseteq[m-1]$. Hence~\eqref{eq:F-FU} and~\eqref{eq:system-FU} imply, %for all $U\subseteq [m-1]$
\begin{eqnarray*}
\forall U\subseteq [m-1],~F_S^U&=&\sum_{V\subseteq[m-1]}u\, D_S^{U,V}(\XX ,F_S)\cdot\bigg(1+\frac{F_S^{m+1}-1}{F_S-1}\sum_{W\subseteq [m-1]}F_S^W\bigg),
 \\ &=& 
F_S^{m+1} \sum_{V\subseteq[m-1]}u\, D_S^{U,V}(\XX ,F_S).
\end{eqnarray*}
Hence by~\eqref{eq:F-FU},
$$F_S=1+u\,F_S^{m+1}\sum_{U,V\subseteq [m-1]}D_S^{U,V}(\XX ,F_S).$$
This together with $\sum_{U,V\subseteq [m-1]}D_S^{U,V}(X,Y)=\frac{X}{1-XY(Y^m-1)/(Y-1)}$ implies Proposition~\ref{prop:GFCatalan} about Catalan arrangements. One can similarly recover  Proposition~\ref{prop:GFsemiorder} about semiorder arrangements by applying  Corollary~\ref{cor:system-FSU} to the set $S'=[-m;m]\setminus\{0\}$ (after observing that for any set $S$, $C_{S\setminus\{0\}}^{U,V}=C_{S}^{U,V}$ and $D_{S\setminus \{0\}}^{U,V}=D_{S}^{U,V}$).
\end{example}

Now, the key lemma for the proof of Theorem~\ref{thm:GF} is the following.
\begin{lemma}\label{lem:rationalD}
For all $U,V\subseteq [m-1]$, the formal power series $D_S^{U,V}(X,Y)$ is a rational function in $\QQ(X,Y)$. This rational function is explicitly computable from $S$, $U$ and $V$.
\end{lemma}
We will prove Lemma~\ref{lem:rationalD} using a transfer-matrix method. Before that, we show that it implies Theorem~\ref{thm:GF}.

\begin{proof}[Proof of Theorem~\ref{thm:GF}]
Corollary~\ref{cor:system-FSU} can be interpreted as giving a system of $2^{m-1}$ linear equations for the $2^{m-1}$ unknown series $\{F_S^U\}_{U\subseteq [m-1]}$ (in the ring $\QQ[u][[x]]$ of formal power series in $x$ with coefficients in $\QQ[u]$). By Lemma~\ref{lem:rationalD} the coefficients of this linear system are rational functions in $u,e^{x},e^{xu}$ and $F_S(x,u)$ which can be computed explicitly from $S$ (in the case $0\in S$, the coefficients do not involve $e^{xu}$). Moreover, observe that $D_S^{U,V}(0,Y)=0$ for all sets $U,V\subseteq [m-1]$, which implies that the specialization $x=0$ of the system reads simply $F_S^U=0$ for all $U\subseteq [m-1]$. This shows that the system of linear equations is independent, hence can be solved by determinants. Doing so, one obtains an equation of the form $F_S^U=g_U(u,e^x,e^{xu},F_S)$ for all $U\subseteq [m-1]$, where $g_U$ is a rational function (in the case $0\in S$, there is no dependence in $e^{xu}$). Furthermore, for all $n\in \NN$, one can use the equation $F_S^U=g_U(u,e^x,e^{xu},F_S)$ to compute the coefficient of $x^n$ in $F_S^U$ from the coefficients $\{[x^k]F_S\}_{k\in[0;n-1]}$ (this fact is deduced from the form of the linear system at the specialization $x=0$). By~\eqref{eq:F-FU}, one gets the equation 
$$F_S=1+\sum_{U\subseteq [m-1]}g_U(u,e^x,e^{xu},F_S)$$ 
which determines $F_S$ uniquely as a series in $\QQ[u][[x]]$ (since this equation can be used to compute its coefficients recursively).
%Hence, injecting the expression of the series $F_S^U$ in~\eqref{eq:F-FU} gives an equation for $F_S$, which allows one to compute the coefficient of $x^n$ in $F_S^U(x,u)$ from the coefficients $\{[x^k]F_S(x,u)$ for all $k<n$. Thus, the equation obtained is a characterization of the series $F_S$. 
If $0\in S$ (resp. $0\notin S$) the above equation has the form $P(F_S(x,u),e^{x},u)$ (resp. $P(F_S(x,u),e^{xu},e^x,u)$) for a non-zero polynomial $P$ with coefficients in $\QQ$, as claimed. 
\end{proof}

It only remains to prove Lemma~\ref{lem:rationalD}.
%
% (with coefficients which are formal power-series in $\QQ[t][[x]]$ defined in terms of $F_S(x,u)$). 
%For now let us show that it implies Theorem~\ref{thm:GF}. By Lemmas~\ref{lem:system-determines-FU} and~\ref{lem:system-determines-FU} $F_S^U(x,u)$ can be expressed as rational function of $F_S(x,u)$, $e^{xt}$ and $t$. Injecting these expressions in~\eqref{eq:F-FU} gives an equation for $F_S$ which can be written in the form $P(F_S(x,u),e^{xt},t)$ for a non-zero trivariate polynomial $P$ with coefficients in $\QQ$. It only remains to prove Lemma~\ref{lem:rationalD}.

\begin{proof}[Proof of Lemma~\ref{lem:rationalD}]
In order to count $S$-shadows we need to consider more general sets (which correspond to ``prefixes'' of $S$-shadows). We call \emph{$S$-shade} a finite set $D\subseteq \NN$ such that $0\in D$ and the graph $G_D$ having vertex set $D$ and edge set
$$E_D:=\{\{i,j\}\mid i,j\in D, ~i-j\in S \}$$
is such that every connected component of $G_D$ contains a vertex in $[\max(D)-m+1;\max(D)]$. Note that the $S$-shadows are the $S$-shades such that $G_D$ is connected.

An important observation is that, for any $S$-shade $D\neq\{0\}$, the set $D'=D\setminus \{\max(D)\}$ is an $S$-shade. Indeed, the graph $G_{D'}$ is obtained from $G_D$ by deleting the vertex $v=\max(D)$, which potentially breaks the connected component of $G_{D}$ containing $v$ in several connected components of $G_D'$, but each of them has a vertex in the interval $[v-m;v-1]$. Therefore one can construct $S$-shades inductively by adding one element at a time. We will now describe this process in detail, and explain how to deduce from it the expression of $D_S^{U,V}(X,Y)$.

Let $D$ be an $S$-shade, and let $\oV=\{0\}\cup\eend(D)=(\max(D)-D)\cap [0;m-1]$. We define $\comp(D)$ as the partition of $\oV$ corresponding to the connected components of $G_D$ (two elements $i,j\in \oV$ are in the same block of $\comp(D)$ if and only if the vertices $\max(D)-i$ and $\max(D)-j$ are connected in $G_D$). Let $r_D=\max(\min(B)\mid B\in\comp(D))\in[0;m-1]$. By definition of $S$-shades, every connected component of $G_D$ has a vertex in $[\max(D)-r_D;\max(D)]$.

It is easy to see that, for a positive integer $d>0$, the set $D^{(d)}:=D\cup \{\max(D)+d\}$ is an $S$-shade if and only if $d\leq m-r_D$. 
Indeed, if $d<m-r_D$ then it is trivial to see that $D^{(d)}$ is an $S$-shade; if $d=m-r_D$ then $D^{(d)}$ is an $S$-shade because $m\in S$ (so that the vertex $v=\max(D)+d$ of $D^{(d)}$ is connected to the connected component of $G_D$ containing the vertex $\max(D)-r_D$); and if $d>m-r_D$ then $D^{(d)}$ is not an $S$-shade because the connected component of $G_{D^{(d)}}$ containing the vertex $\max(D)-r_D$ is not connected to any vertex in $[\max(D^{(d)})-m+1;\max(D^{(d)})]$. 

Let $\vD_S$ be the set of all $S$-shades. Let $\mG_S$ be the infinite directed graph with vertex set $\vD_S$ and edge set 
$$\mE_S:=\{(D,D^{(d)})\mid D\in \vD_S, ~d\in [m-r_D]\}.$$
From the above discussion it is clear that the set or $S$-shades is in bijection with the set of directed paths in $\mG_S$ starting at the vertex $D_0:=\{0\}$.
The $S$-shadows are in bijection with the directed paths in $\mG_S$ starting at $D_0:=\{0\}$ and ending at a vertex in $\mD_S$ (that is, ending at an $S$-shade $D$ such that partition $\comp(D)$ has a single block). 
Let us now fix $U,V\subseteq [m-1]$ and consider the set $\mD_S^{U,V}$ of $S$-shadows $D$ such that $\start(D)=U$ and $\eend(D)=V$. The set $\mD_S^{U,V}$ is is bijection with a set of paths in $\mG_S$ starting at $D_U:=\{0\}\cup U\in \vD_S$, but one needs to be mindful that not every path starting at $D_U$ corresponds to $S$-shade $D$ satisfying $\start(D)=U$. To be precise, among the paths of positive lengths starting at $D_U$, the paths in correspondence with $S$-shades satisfying $\start(D)=U$ are those starting with an edge of the form $(D_U,D_U^{(d)})$ with $d\geq m-\max(U)$. 
This gives the following result, which we encapsulate for future reference.\\
\textbf{Fact:}
%\begin{lemma} \label{lem:DSUV-aspaths}
Let $U,V\subseteq [m-1]$, let $D_U:=\{0\}\cup U$ and let $D_V:=\{0\}\cup V$.
The set $\mD_S^{U,V}\setminus\{D_U\}$ is in bijection with the set of paths in the graph $\mG_S$ starting at one of the vertices $D_U^{(d)}$ for $d$ in $[m-\max(U); m-r_{D_U}]$ and ending at a vertex $D\in \vD_S$ such that the set partition $\comp(D)$ has a single block which is the set $D_V$.
%\end{lemma}

We will now embed the infinite graph $\mG_S$ in a finite graph $\oG_S$. The main observation is that for an $S$-shade $D$ and $d\in [m-r_D]$, the partition $\comp(D^{(d)})$ is uniquely determined by $\comp(D)$ and $d$. Indeed,
%$$\eend(D^{(d)}):=((d+\eend(D))\cap[m-1])\cup \{d\},$$
%and 
if the partition $\comp(D)$ has blocks $B_1,\ldots,B_k$, then upon denoting $B_i^{(d)}=(d+B_i)\cap [m-1]$ for all $i\in [k]$, the partition $\comp(D^{(d)})$ is obtained from the partition $\{\{0\},B_1^{(d)},B_2^{(d)},\ldots,B_k^{(d)}\}$ 
%\cup \{\{d+B\}\cap [m-1]\mid B\in \comp(D)\}$$
by merging the block $\{0\}$ with all the blocks containing an element in $S$.

For a set $V$, let $\mB(V)$ be the set of set partitions of $V$.
Let 
$$\oD_m:=\bigcup_{V\subseteq [m-1]}\mB(\{0\}\cup V).$$
For a partition $\oD=\{B_1,\ldots,B_k\}\in \oD_m$, we define $r_\oD=\max(\min(B_1),\ldots,\min(B_k))$. Then, for all $d\in [m-r_{\oD}]$, upon denoting $B_i^{(d)}:=(d+B_i)\cap [m-1]$ for all $i\in [k]$, we define $\oD^{(d)}\in \oD_m$ as the partition obtained from the partition 
$\{\{0\},B_1^{(d)},B_2^{(d)},\ldots,B_k^{(d)}\}$
by merging the block $\{0\}$ with all the blocks containing an element in $S$.
Let $\oG_S$ be the finite directed graph with vertex set $\oD_m$ and edge set 
$$\oE_S=\{(\oD,\oD^{(d)})\mid \oD\in \oD_m, ~d\in [m-r_\oD]\}.$$
%where $r_\oD=\max(\min(B)\mid B\in\oD)$, and $\oD^{(d)}\in \oD_m$ is obtained from the partition 
%$$\{0\} \cup \{\{d+B\}\cap [m-1]\mid B\in \oD\}$$
%by merging the block $\{0\}$ with all the blocks containing an element in $S$.

From the above discussion it is clear that for all $k\in\NN$, $D\in \vD_S$ and $\oD'\in\oD_m$, the set of paths of length $k$ in $\mG_S$ from the $S$-shade $D$ to an $S$-shade $D'$ such that $\comp(D')=\oD'$ is in bijection with the set of paths of length $k$ in $\oG_S$ from $\comp(D)$ to $\oD'$. In particular, the set of $S$-shadows is in bijection with the set of paths in $\oG_S$ starting at the set partition $\{\{0\}\}$ and ending at any set partition in $\oD_m$ consisting of a single block. More generally, for any subsets $U,V\subseteq [m-1]$, the Fact stated above implies that the set $\mD_S^{U,V}\setminus\{D_U\}$ is in bijection with the set of paths in the graph $\oG_S$ starting at one of the vertices $\comp(D_U^{(d)})$ for $d$ in $[m-\max(U); m-r_{D_U}]$, and ending at the single-block partition $\{D_V\}$. 
%Let us now fix $U,V\susbeteq [m-1]$ and consider the set $\mD_S^{U,V}$ of $S$-shadows $D$ such that $\start(D)=U$ and $\eend(D)=U$. The set $\mD_S^{U,V}$ can be put in correspondence with a set of paths in $\mG_S$ starting at $D_U:=\{0\}\cup U\in \vD_S$, but one needs to be mindful that not all paths starting at $D_U$ correspond to $S$-shade $D$ satisfying $\start(D)$
This readily implies
\begin{equation}\label{eq:GF-as-path}
D_S^{U,V}(X,Y)= %\sum_{D\in \mD_S^{U,V}}X^{|D|}Y^{\max(D)} = 
 X^{1+|U|}Y^{\max(U)}\bigg(\one_{\eend(D_U)=V}+X\,\sum_{d=m-\max(U)}^{m-r_{D_U}}
 \sum_{\substack{P\textrm{ path in }\oG_S\\ \textrm{ starting at }\comp(D_U^{(d)})\\ \textrm{and ending at }\{D_V\} }}w(P)\bigg),
\end{equation}
%\begin{equation}\label{eq:GF-as-path}
%\forall V\subseteq[m-1],~\sum_{\substack{D\in \mD_S\\ \eend(D)=V}}X^{|D|}Y^{\max(D)}=X\,\sum_{\substack{P\textrm{ path in }\oGS\\ \textrm{starting at }\{\{0\}\}}w(P)
%\end{equation}
where $w(P)=\prod_{e\textrm{ edge of the path }P}w(e)$ and the \emph{weight} $w(e)$ of an edge $e\in \oE_S$ of the form $(\oD,\oD^{(d)})$ is $w(e)=XY^d$.% (and $\one_{\eend(D_U)=V}$ is equal to 1 if $\eend(D_U)=V$, and 0 otherwise).
 
Finally, the transfer-matrix formula expresses the right-hand side of~\ref{eq:GF-as-path} as a rational function in $X$ and $Y$. Precisely, one can consider the weighted adjacency matrix $A_S=(a_{\oD,\oD'})_{\oD,\oD'\in \oD_m}$ of the graph $\oG_S$, whose entry $a_{\oD,\oD'}$ is equal to $XY^d$ if $\oD'=\oD^{(d)}$ for some $d\in [m-r_\oD]$, and equal to 0 otherwise. Then, for all $\oD,\oD' \in \oD_m$, the sum 
$$\sum_{\substack{P\textrm{ path in }\oG_S\\ \textrm{ starting at }\oD \textrm{ and ending at }\oD'}}w(P)$$
 is equal to the entry $\oD,\oD'$ in the matrix $(\Id-A_S)^{-1}$. In particular, the expression of this matrix entry as a ratio of determinants shows that the above sum is a rational function in $X$ and $Y$. Moreover the matrix $A_S$ is computable from $S$. Thus, for all $U,V\subseteq [m-1]$ the series $D_S^{U,V}(X,Y)$ is a rational function of $X,Y$ which can be computed from $S$. This completes the proof of Lemma~\ref{lem:rationalD}, and Theorem~\ref{thm:GF}. 
\end{proof}

\begin{remark}
For a set $S\subseteq \ZZ$ satisfying~\eqref{eq:condition-sym-trans}, let $F_{S,k}(x)=[u^k]F_S(x,u)$ be the generating function of $k$-dimensional faces of the arrangements $\mA_S^n$. By Theorem~\ref{thm:main} $F_{S,k}(x)$ is the generating function of the marked trees in $\oT(S)$ having $k$ blocks. It is easy to see from the above proof that $F_{S,1}$ is a rational function of $e^x$. Precisely, 
$$F_{S,1}(x)=\sum_{U,V\subseteq [m-1]}D_S^{U,V}(e^x-1,1)-\one_{0\notin S}(e^x-1-x).$$
With a bit more work, one could similarly show that for all $k\in \NN$, $F_{S,k}(x)$ is a rational function of $e^x$ computable from $S$.
\end{remark}

%%%%%%%%%%%%%%%%%%%%%%%%%%%%%%%%%%%%%%%%%%%%%%%%%%%%%%%%%%%%%%%%%%%%%%%%%%%%%%%%%%

\section{Proof of Theorem~\ref{thm:main}}\label{sec:proof}
In this section we prove Theorem~\ref{thm:main}.
Roughly speaking, the proof consists in seeing the faces of an arrangement $\mA\subseteq \mA_m^n$ as regions of the restrictions of $\mA$ (to affine subspaces in the intersection lattice), and applying Theorem~\ref{thm:regions} to these arrangements. 
Hence this proof strategy first calls for identifying each restriction of the strongly transitive arrangement $\mA$ with a transitive arrangement.

Recall that the \emph{intersection lattice} $\mL(\mA)$ of a hyperplane arrangement $\mA\subset \RR^n$ is the set of non-empty intersections of hyperplanes in $\mA$:
$$\mL(\mA):=\bigg\{\bigcap_{i=1}^k H_i~\bigg|~ k\geq 0,H_1,\ldots,H_k\in\mA~\textrm{ such that }\bigcap_{i=1}^k H_i\neq \emptyset\bigg\}.$$
For $L\in \mL(A)$ we denote by $\mA_L$ the restriction of $\mA$ to the affine vector space $L$. 

Consider an arrangement $\mA\subseteq \mA_m^n$ and an affine subspace $L\in\mL(\mA)$ of dimension $d$. We want to identify the arrangement $\mA_L$ with an arrangement $\wA_L\subseteq \mA_{mn}^d$. 
If $L$ is contained in a hyperplane of the form $\{x_i-x_j=s\}$ we write $i\Lsim j$. This is an equivalence relation, and we denote by $\blocks(L)$ the set of equivalence classes. Let $\{B_1,\ldots,B_d\}=\blocks(L)$. 
For concreteness of notation, we adopt the convention that $\min(B_1)<\min(B_2)<\cdots<\min(B_d)$. Next, we define a linear order $\precL$ within each block $B_k$ as follows: for $i,j\in B_k\subseteq[n]$, we set $i\precL j$ if for any point $(x_1,\ldots,x_n)$ in $L$ one has $x_i<x_j$ or ($x_i=x_j$ and $i<j$). 
%For all $i\in [n]$, we denote by $i_L$ the minimal element (for the order $\precL$) of the block $B_k$ containing $i$, and we denote by $\de_L(i)\in \NN$ the value of $x_i-x_{i_L}$ inside the affine space $L$. 
For all $k\in [d]$, and all $i\in B_k$ we define $\de_L(i)$ to be the value of f $x_i-x_{a_k}$ for any point $(x_1,\ldots,x_n)$ in $L$, where $a_k$ the minimal element (for the order $\precL$) of the block $B_k$. Note that $\de_L(i)\geq 0$ for all $i\in[n]$. 
Finally, we define the arrangement $\wA_L\subseteq \RR^d$ as follows: 
\begin{equation}\nonumber %\label{eq:AL}
\wA_L:=\bigcup_{\substack{k,\ell\in[d]\\ k\neq \ell}} ~~\bigcup_{\substack{ i\in B_k,~ j\in B_\ell,~s\in\ZZ\\ \{x_i-x_j=s\}\in \mA}} \{x_k-x_\ell=s-\de_L(i)+\de_L(j)\}.
\end{equation}
%$$\wA_L:=\bigcup_{\substack{k,\ell\in[d]\\ k\neq \ell}} ~~\bigcup_{\substack{(i,j,s)\in\trmn(\mA)\\ i\in B_k,~ j\in B_\ell}} \{x_k-x_\ell=s-\de_L(i)+\de_L(j)\}.$$

%The following lemma is easy to check.
\begin{lemma}
Let $\mA\subseteq \mA_m^n$, and let $L\in\mL(A)$.
The arrangements $\mA_L$ and $\wA_L$ are isomorphic. Moreover, $\wA_L\subseteq \mA_{mn}^d$, where $d=\dim(L)$.
\end{lemma}

\begin{proof}
Let $\{B_1,\ldots,B_d\}=\blocks(L)$. It is clear that $d=\dim(L)$. Consider the affine transformation 
$\psi_L:L\to \RR^d$ defined by 
\begin{equation}
\psi_L(x_1,\ldots,x_n)=(x_{a_1},\ldots,x_{a_d}),\label{eq:defpsiL}
\end{equation}
where for all $k\in [d]$, $a_k\in[n]$ is the smallest element in $B_k$ for the order $\precL$. 
Note that $\psi_L$ is invertible since for any point $(x_1,\ldots,x_n)$ in $L$, the coordinate $x_i$ is equal to $x_{a_k}+\de_L(i)$ if $i$ is in the block $B_k$. Hence $\psi_L$ is bijective.
 Moreover, a hyperplane $H=\{x_i-x_j=s\}$ of $\mA$ corresponds to a hyperplane of the restriction $\mA_L$ (equivalently, $\emptyset\neq H\cap L\neq L$) if and only if $i\in B_k$ and $j\in B_\ell$ for some $k\neq \ell$. In this case, 
 $$\psi_L(H\cap L)=\{x_k-x_\ell=s-\de_L(i)-\de_L(j)\},$$
 since $x_i=x_{a_k}+\de_L(i)$ and $x_j=x_{a_\ell}+\de_L(j)$ in $L$. So the image of $\mA_L$ by the affine isomorphism $\psi_L$ is $\wA_L$. Lastly, it is clear that $\de_L(i)\leq (n-1)m$ for all $i\in [n]$, hence $\wA_L\subseteq \mA_{mn}^d$.
\end{proof}

%OB: To be concrete in $\psi$ we should specify an order on the blocks of $\blocks(L)$. For instance we could assume $\min(B_1)<\cdots<\min(B_d)$.

\begin{lemma}\label{lem:strongly-transitive-good}
If a braid-type arrangement $\mA$ is strongly transitive, then for any affine subspace $L$ in $\mL(\mA)$, the arrangement $\wA_L$ is strongly transitive.
\end{lemma}

\begin{proof}
We can reason by induction on the codimension of $L$, so it suffices to treat the case where $L$ is a hyperplane of the form $\{x_a-x_b=r\}$ for some $(a,b,r)\in \trmn(\mA)$. Furthermore, it is clear that for every permutation $\pi\in\fS_n$, the arrangement $\pi(\mA)$ is strongly transitive, hence we can assume $a=n$ and $b=n-1$. Hence $L=\{x_{n}=x_{n-1}+r\}\in\mA$ for some $r\geq 0$, and in the above notation we have
\bitem 
\item $d=n-1$ and $\blocks(L)=\{B_1,\ldots,B_{n-1}\}$ with $B_i=\{i\}$ for all $i\in[n-2]$ and $B_{n-1}=\{n-1,n\}$,
\item $\de_L(i)=0$ for all $i\in[n-1]$, and $\de_L(n)=r$.
%$$\wA_L=\left(\bigcap_{\substack{(i,j,s)\in\trmn(\mA)//i,j\leq n-1}}\{x_i-x_j=s\}\right)\cap\left(\bigcap_{i,n}\right)$$
\eitem

%We claim that for all $s\in\{0,\ldots, r\}$ and all $i\in [n-2]$ the hyperplane $\{x_i-x_{n-1}=s\}$ is in $\wA_L$. Indeed, if we suppose for contradiction that $\{x_i-x_{n-1}=s\}\notin \wA_L$, then we get $\{x_i-x_{n-1}=s\}\notin \mA$ and $\{x_n-x_{i}=r-s\}\notin \mA$. Since $\mA$ is strongly transitive, this implies $L=\{x_n-x_{n-1}=r\}\notin \mA$, which is a contradiction.

Now let $i,j,k$ be distinct integers in $[n-1]$ and let $s,t\in \NN$ be such that 
$$\{x_i-x_j=s\}\notin \wA_L\quad \textrm{ and }\quad \{x_j-x_k=t\}\notin \wA_L.$$ 
We want to show that $\{x_i-x_k=s+t\}\notin \wA_L$. If $i,j,k\in[n-2]$, then $\{x_i-x_j=s+t\}\notin \wA_L$ since $\mA$ is strongly transitive. We will now treat the cases $j=n-1$, $i=n-1$, and $k=n-1$ in this order.\\[1mm]
\ni \emph{Case $j=n-1$:} In this case we have $\{x_{i}-x_{n-1}=s\}\notin \mA$ and $\{x_{n-1}-x_{k}=t\}\notin \mA$. Since $\mA$ is strongly transitive, we get $\{x_i-x_k=s+t\}\notin\mA$. This implies $\{x_i-x_k=s+t\}\notin\wA_L$ as wanted.\\[1mm]
\ni \emph{Case $i=n-1$:} Since $\{x_{n-1}-x_j=s\}\notin \wA_L$, we deduce $\{x_{n-1}-x_j=s\}\notin \mA$ and $\{x_n-x_j=s+r\}\notin \mA$. 
Since $\mA$ is strongly transitive, we get $\{x_{n-1}-x_k=s+t\}\notin \mA$ and $\{x_n-x_k=s+t+r\}\notin \mA$. This gives $\{x_{n-1}-x_k=s+t\}\notin \wA_L$ as wanted.\\[1mm]
\ni \emph{Case $k=n-1$:}  Since $\{x_{j}-x_{n-1}=t\}\notin \wA_L$, we deduce $\{x_j-x_{n-1}=t\}\notin \mA$ and $\{x_j-x_n=t-r\}\notin \mA$. Now we come to the important observation: for  all $r'\in[0;r]$ the hyperplane $\{x_j-x_{n-1}=r'\}$ is in $\wA_L$. Indeed, if we suppose for contradiction that $\{x_j-x_{n-1}=r'\}\notin \wA_L$, then we get $\{x_n-x_{j}=r-r'\}\notin \mA$ and $\{x_j-x_{n-1}=r'\}\notin \mA$, thus $\{x_n-x_{n-1}=r\}\notin \mA$ (because $\mA$ is strongly transitive) which is a contradiction.
From the above observation we conclude $t> r$.
Since $\mA$ is strongly transitive, we get $\{x_i-x_{n-1}=s+t\}\notin \mA$ and $\{x_i-x_n=s+t-r\}\notin \mA$. This gives $\{x_i-x_{n-1}=s+t\}\notin \wA_L$ as wanted.
\end{proof}

We can now describe the proof of Theorem~\ref{thm:main} more precisely. We want to prove that, for any strongly transitive arrangement $\mA$, the map $\bPhi_\mA$ is a bijection between the set $\Tmn(\mA)$ of marked trees and the set $\mF(\mA)$ of faces. We introduce two related sets:
\begin{eqnarray*}
\wt{\mT}(\mA)&:=&\left\{\left(L,\wt{T}\right)\mid L\in \mL(\mA),~\wt{T}\in\mT_{mn}^{\dim(L)}(\wA_L)\right\},\\
\wt{\mF}(\mA)&:=&\left\{\left(L,\wt{R}\right)\mid L\in \mL(\mA),~\wt{R}\in \mR(\wA_L)\right\},
\end{eqnarray*}
where $\mR(\wA_L)$ is the set of regions of $\wA_L$.
Our proof of Theorem~\ref{thm:main} will consist in establishing the commutative diagram of bijections represented in Figure~\ref{fig:diagramproof}.
\begin{figure}[h!]%\begin{center}
\[
\begin{tikzcd}
(T,\mu)\in\oTmn(\mA) \arrow[d, "\textrm{bijection }\Gamma"']\arrow[rrr,"\bPhi_\mA"] &&& F\in \mF(\mA)\\
\left(L,\wt{T}\right)\in\wt{\mT}(\mA) \arrow[rrr,"\textrm{bijection }\wt{\Phi}_\mA"'] &&& \left(L,\wt{R}\right)\in \wt{\mF}(\mA)\arrow[u, "\textrm{bijection }\Theta"]
\end{tikzcd}
\]
\caption{Commutative diagram representing the proof of Theorem~\ref{thm:main}.}\label{fig:diagramproof}
\end{figure}

\newcommand{\wTmA}{\wt{\mT}(\mA)}

We start by defining $\wt{\Phi}_\mA:\wt{\mT}(\mA)\to\wt{\mF}(\mA)$ as the map which associates to each pair $(L,\wt{T})\in\wTmA$ the pair $(L,\Phi_{\wA_L}(T))\in \wt{\mF}(\mA)$. 
By combining Lemma~\ref{lem:strongly-transitive-good} with Theorem~\ref{thm:regions} from~\cite{OB}, we deduce that $\wt{\Phi}_\mA$ is a bijection between the sets $\wt{\mT}(\mA)$ and $\wt{\mF}(\mA)$.
 It is also clear that the sets $\wt{\mF}(\mA)$ and $\mF(\mA)$ are in bijection. Indeed, the faces of $\mA$ (of dimension $d$) correspond to the regions of the restrictions of $\mA$ to subspaces (of dimension $d$) in the intersection lattice $\mL(\mA)$. Precisely, the bijection $\Theta:\wt{\mF}(\mA)\to\mF(\mA)$ associates to each pair $(L,\wt{R})\in \wt{\mF}(\mA)$ the face $\psi_L^{-1}(\wt{R})\subseteq L\subseteq \RR^n$ of $\mA$, where $\psi_L$ is the affine linear transformation defined by~\eqref{eq:defpsiL}. 
 
It remains to describe the bijection $\Gamma$ between $\oTmn(\mA)$ and $\wt{\mT}(\mA)$, and to check that the diagram in Figure~\ref{fig:diagramproof} is commutative.
Before describing the bijection $\Gamma$, we need to establish a technical property satisfied by the arrangements of the form $\wA_L$.

%
%
%The previous lemma together with Theorem~\ref{thm:regions} from \cite{OB} immediately give the following result.
%% to encode the regions of $\mA_L$ (that is, the faces of $\mA$ of dimension $\dim(L)$ inside $L$) by some trees in $\oTmn(\mA)$. 
%\begin{cor}\label{cor:facesL}
%Let $\mA\subseteq\mA_m^n$ be a strongly transitive arrangement, and let $L\in \mL(\mA)$ be a subspace. The faces of $\mA$ of dimension $d=\dim(L)$ contained in $L$ are in bijection with the the set of trees $\Tmnd(\wA_L)\subseteq\mT_{mn}^d$. The bijection is given by $\psi_L^{-1}\circ\Phi_{\wA_L}$. 
%\end{cor}

% Before we move to the next step of the proof, we establish an additional property satisfied by the arrangements $\wA_L$.

\begin{lemma}\label{lem:cadet-condition-simple}
Let $\mA$ be a strongly transitive arrangement. Let $L\in\mL(\mA)$ and let $\{B_1,\ldots,B_d\}=\blocks(L)$. Let $\ell\in [d]$ and let $\de^\ell=\max(\de_L(j)\mid j\in B_\ell)$. Then, for all $k\in[d]\setminus\{\ell\}$ and all $t\in[0; \de^\ell]$, one has $\{x_k-x_\ell=t\}\in \wA_L$ unless $|B_k|=|B_\ell|=1$.
\end{lemma}

\begin{proof} 
Let $k\in[d]\setminus\{\ell\}$ and let $t\in[0; \de^\ell]$. We first assume that $|B_\ell|>1$, and prove that $\{x_k-x_\ell=t\}\in \wA_L$ in this case. 

Let $a_k$ and $a_\ell$ be the minimal elements of $B_k$ and $B_{\ell}$ respectively, for the order $\preceqL$, and let  $b_\ell$ be the maximal element of $B_\ell$. By definition of $\prec_L$, we have  $\de_L(a_k)=\de_L(a_\ell)=0$ and $\de_L(b_\ell)=\de^\ell$.

Recall that $i,j\in[n]$ are in the same block of $\blocks(L)$ if and only if the subspace $L$ is contained in a hyperplane of the form $\{x_{i}-x_{j}=s\}$, and that in this case $s=\de_L(i)-\de_L(j)$. In particular $L$ is contained in the hyperplane $\{x_{b_\ell}-x_{a_\ell}=\de^\ell\}$. 
Since the subspace $L$ is in $\mL(A)$, it is an intersection of hyperplanes in $\mA$ which are all of the form $\{x_i-x_j=\de_L(i)-\de_L(j)\}$. From the inclusion 
$L\subseteq \{x_{b_\ell}-x_{a_\ell}=\de^\ell\}$,
we conclude that there exist indices $b_\ell=i_1,i_2,\ldots,i_m=a_\ell$ in $B_\ell$ such that $\{x_{i_j}-x_{i_{j+1}}=\de_L(i_j)-\de_L(i_{j+1})\}$ is in $\mA$.
In particular, there must exist $i,j$ among those indices such that $\de_L(j)\leq t\leq \de_L(i)$ and $\{x_i-x_j=\de_L(i)-\de_L(j)\}\in \mA$.
%
%
%Since $a_\ell\Lsim b_\ell$ the subspace $L\in\mL(A)$ is contained in a hyperplane of the form $\{x_{b_\ell}-x_{a_\ell}=s\}$ for some $s$ (in fact $s=\delta_L(b_\ell)-\delta_L(a_L)=\de^L$. 
%Since $L\in\mL(A)$ and $\de_L(a_\ell)\leq t\leq \de_L(b_\ell)$, the subspace $L$ is contained in a hyperplane of $\mA$ of the form 
%$\{x_i-x_j=s''\}$
%with $i,j\in B_\ell$ and $\de_L(j)\leq t\leq \de_L(i)$. Note also that in this case $s''=\de_L(i)-\de_L(j)$, hence
%$$\{x_i-x_j=\de_L(i)-\de_L(j)\}\in \mA.$$ 

Now, supposing for contradiction that $\{x_k-x_\ell=t\}\notin \wA_L$, one gets $\{x_{i}-x_{a_k}=\de_L(i)-t\}\notin \mA$ and $\{x_{a_k}-x_{j}=t-\de_L(j)\}\notin \mA$. 
Since $\mA$ is strongly transitive, this implies $\{x_i-x_j=\de_L(i)-\de_L(j)\}\notin \mA$ which is a contradiction. This concludes the proof in the case $|B_\ell|>1$. 

In the case $|B_\ell|=1$, we have $t=\de^\ell=0$. By reversing the role of $k$ and $\ell$, the previous case shows that if $|B_k|>1$ then $\{x_k-x_\ell=0\}\in \wA_L$ as wanted. Finally, in the case $|B_k|=|B_\ell|=1$ there is nothing to prove.
\end{proof}

Our next step is to encode the elements of the intersection lattice $\mL(\mA)$.
\begin{definition}\label{def:A-connected-partition}
Let $\mA\subseteq \RR^n$ be a braid-type arrangement. Let $B\subseteq [n]$ be a set, and let $\de:[n]\to \NN$ be a map. The pair $(B,\de)$ is said to be \emph{$\mA$-connected} if the graph $G$ with vertex set $B$ and edge set 
$$E:=\{\{i,j\}\mid i,j\in B,~\{x_i-x_j=\de(i)-\de(j)\}\in \mA\}$$ 
is connected. 

We define $\mP(\mA)$ as the set of pairs $(\{B_1,\ldots,B_d\},\de)$, where 
\begin{compactitem}
\item $\{B_1,\ldots,B_d\}$ is a set partition of $[n]$, and
\item $\de:[n]\to \NN$ is a map such that, for all $k\in[d]$, $\min(\de(i)\mid i\in B_k)=0$ and $(B_k,\de)$ is $\mA$-connected.
\end{compactitem}
\end{definition}

\begin{lemma}\label{lem:code-subspace}
Let $\mA\subseteq \RR^n$ be a braid-type arrangement. The intersection lattice $\mL(\mA)$ is in bijection with the set $\mP(\mA)$.
 The bijection $\Lambda:\mL(\mA)\to\mP(\mA)$ associates to each affine subspace $L\in\mL(A)$ the pair $\Lambda(L)=(\blocks(L),\de_L)$. The inverse bijection is the map $\Delta:\mP(\mA)\to \mL(\mA)$ defined by 
$$\Delta(\{B_1,\ldots,B_d\},\de)=\bigcap_{k=1}^d\bigcap_{\substack{(i,j,s)\in \trmn(\mA)\\ i,j\in B_k,~s=\de(i)-\de(j)}}\{x_i-x_j=s\}.$$
\end{lemma}

\begin{proof}
Let us first check that for any subspace $L$ in $\mL(\mA)$, the pair $(\blocks(L),\de_L)$ is in $\mP(\mA)$. The only non-trivial point is to check that every block of the set partition $\{B_1,\ldots,B_d\}=\blocks(L)$ is $\mA$-connected. Let $k\in[d]$, and let $G=(B_k,E)$ be the graph associated to $(B_k,\de_L)$ in Definition~\ref{def:A-connected-partition}. Let $i,j\in B_k$. Since $i,j$ are in the same block of the partition $\blocks(L)$, we know that $L$ is included in a hyperplane of the form $\{x_i-x_j=s\}$ for some $s\in \ZZ$. Since $L\in\mL(\mA)$, this inclusion implies the existence of some indices $i=i_0,i_1,\ldots,i_q=j$ and some integers $s_1,\ldots,s_q$ such that 
$$L\subseteq \{x_{i_{p-1}}-x_{i_p}=s_p\}\in \mA$$ 
for all $p\in[q]$. Hence, the indices $i=i_0,i_1,\ldots,i_q=j$ are all in $B_k$ and the edge $\{i_{p-1},i_p\}$ is in $E$ for all $p\in[q]$. This shows that $i$ and $j$ are connected by the path $i_0,i_1,\ldots,i_{q}$ in $G$. Hence every block of the set partition $\{B_1,\ldots,B_d\}=\blocks(L)$ is $\mA$-connected, and $\Lambda(L)=(\blocks(L),\de_L)$ is in $\mP(\mA)$. 

Next we note that the map $\Lambda:\mL(\mA)\to \mP(\mA)$ is injective. Indeed, if it clear that an affine space $L\in \mL(\mA)$ is contained in a hyperplane $\{x_i-x_j=s\}$ of $\mA$ if and only if $i,j$ are in the same block of the partition $\blocks(L)$ and $s=\de_L(i)-\de_L(j)$. Thus $L$ can be recovered from $\Lambda(L)=(\blocks(L),\de_L)$.

Lastly, one can check that $\Lambda\circ\De=\Id_{\mP(\mA)}$. Indeed, for a pair $(\{B_1,\ldots,B_d\},\de)\in \mP(\mA)$, and $L=\De(\{B_1,\ldots,B_d\},\de)$, the assumption that each block $B_k$ is $\mA$-connected ensures that $\blocks(L)=\{B_1,\ldots,B_d\}$, and it follows without difficulty that $\Lambda(L)=(\{B_1,\ldots,B_d\},\de)$. 

Thus, $\Lambda$ is a bijection, and $\De$ is the inverse bijection.
\end{proof}

We are ready to describe the bijection $\Gamma$ from $\oTmn(\mA)$ to $\wTmA$. Given a marked tree $(T,\mu)\in\oTmn(\mA)$, we consider the set partition $\{B_1,\ldots,B_d\}=\blocks(\mu)$. 
For concreteness of notation, we adopt the convention that $\min(B_1)<\min(B_2)<\cdots<\min(B_d)$. 
Recall that each block $B_k$ corresponds to a path of marked edges in $(T,\mu)$ (each marked edge being between a node and its cadet child). Let $a_k\in B_k$ be the node of $T$ which is the ancestor of all the other nodes in $B_k$. We define a map $\de_\mu:[n]\to \NN$ by setting $\de_\mu(i)=\driftT(i)-\driftT(a_k)$ for all $i\in B_k$. By definition, each block $B_k$ is $\mA$-connected (see Definition~\ref{def:A-connected-tree}), which is equivalent to the fact that the pair $(B_k,\de_\mu)$ is $\mA$-connected (see Definition~\ref{def:A-connected-partition}). Thus, the pair $(\{B_1,\ldots,B_d\},\de_\mu)$ is in $\mP(\mA)$, and corresponds to a subspace $L=\Delta(\{B_1,\ldots,B_d\},\de_\mu)\in\mL(\mA)$ via the bijection of Lemma~\ref{lem:code-subspace}. 
Observe that, by Remark~\ref{rk:eqL}, 
$$L=\bigcap_{\substack{\{i,j\}\in \mu \\ i=s\child(j)}}\{x_i-x_j=s\}.$$
Now consider the tree $\wt T$ obtained from $T$ by 
\bitem
\item[1.] contracting all the marked edges: for all $k\in[d]$, the marked path of $T$ corresponding to $B_k$ is replaced by a node of $\wT$ labeled $k$, 
\item[2.] adding leaves as right children of each node of $\wT$ so as to get a total of $mn+1$ children for each node. 
\eitem
We define $\Gamma(T,\mu):=(L,\wt T)$. We will now proceed to check that $\wt T$ is in $\Tmnd(\wA_L)$, (which is where Lemma~\ref{lem:cadet-condition-simple} will be used), and that any tree in $\Tmnd(\wA_L)$ is obtained in this manner.

\fig{width=\linewidth}{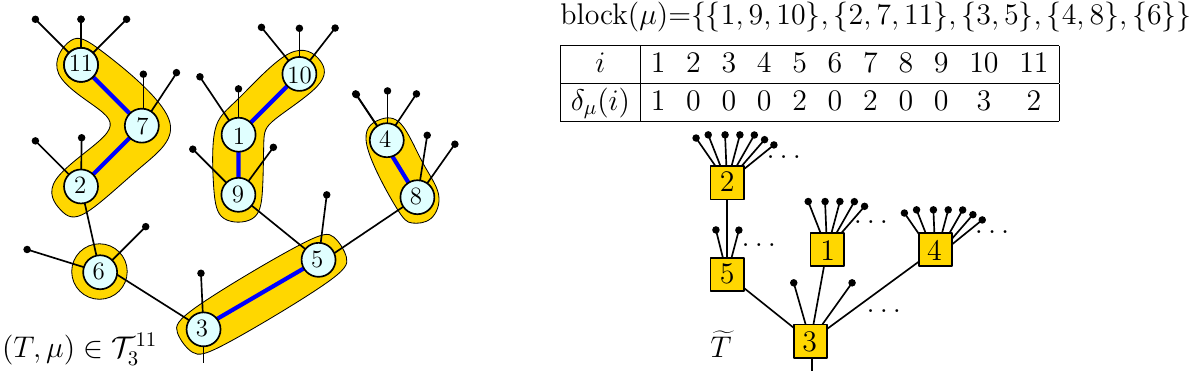}{The bijection $\Gamma:\oTmn(\mA)\to\wTmA$ associating to each marked tree $(T,\mu)\in \oTmn(\mA)$ the pair $(\wt T,L)$, where $L=\Delta(\blocks(\mu),\de_\mu)\in \mL(\mA)$.}

\begin{lemma} \label{lem:contraction-tree}
For any marked tree $(T,\mu)$ in $\oTmn(\mA)$, the pair $\Gamma(T,\mu)=(L,\wt T)$ is in $\wTmA$. 
\end{lemma}

\begin{proof}
Let $(T,\mu)$ be in $\oTmn(\mA)$, and let $L=\De(\{B_1,\ldots,B_d\},\de)\in\mL(\mA)$ be defined as above. We already know that $L$ is in $\mL(\mA)$, and it is clear that $\wt T$ is in $\mT_{mn}^d$. It remains to show that $\wt T$ is in $\Tmnd(\wA_L)$.
What we need to show is that for any triple $(k,\ell,t)\in \tr_{mn}^d$ such that $k=t\cadet(\ell)$ in $\wt T$, the hyperplane $\{x_k-x_\ell=t\}$ is in $\wA_L$. 

Let $(k,\ell,t)\in \tr_{mn}^d$ be such that $k=t\cadet(\ell)$ in $\wt T$. As above, we let $a_k\in B_k$ be the ancestor in $T$ of all the nodes in $B_k$. Since $k=t\cadet(\ell)$ in $\wt T$, the parent $j$ of the node $a_k$ in $T$ is in $B_\ell$. This situation is represented in Figure~\ref{fig:condition-tree-proof}(a).

\fig{width=\linewidth}{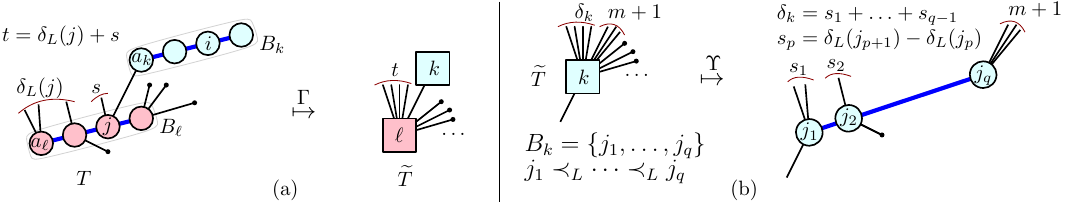}{(a) Proof of Lemma~\ref{lem:contraction-tree}: the vertices $k$ and $\ell$ of the tree $\wT$ and the vertices $i$, $j$ and $a_k$ in the tree $T$. (b) Definition of the map $\Ups$: the node $k$ of $\wT\in\Tmnd(\wA_L)$ and the corresponding path of marked edges in $T\in\oTmn(\mA)$.}

In the case where the vertex $j$ of $T$ is not the descendant of all the other vertices in $B_\ell$, one has $t<\de^\ell=\max(\de_L(j')\mid j'\in B_\ell)$. Hence, in this case Lemma~\ref{lem:cadet-condition-simple} implies that $\{x_k-x_\ell=t\}$ is in $\wA_L$, as wanted. 

We now treat the case where the vertex $j$ is the descendant of all the other vertices in $B_\ell$. In this case, $a_k=s\cadet(j)$ in $T$, where $s=t-\de_L(j)$. Since $(T,\mu)$ belongs to $\oTmn(\mA)$, at least one of the following situations holds:
\begin{compactitem}
\item[(a)] there exists $i'\musim a_k$ and $j'\musim j$ such that $\{x_{i'}-x_{j'}=\driftT(i')-\driftT(j')\}$ is in $\mA$, or 
\item[(b)] $s=0$ and $a_k<j$. 
\end{compactitem}
In situation (a), the hyperplane $\{x_{i'}-x_{j'}=t+\de_L(i')-\de_L(j')\}$ is in $\mA$, since 
$$\driftT(i')-\driftT(j')=(\driftT(a_k)+\de_L(i'))-(\driftT(a_\ell)+\de_L(j'))=t+\de_L(i')-\de_L(j').$$ 
Thus, the hyperplane $\{x_k-x_\ell=t\}$ is in $\wA_L$, as wanted.

In situation (b), one has $t=\de_L(j)=\max(\de_L(j')\mid j'\in B_\ell)$. Hence Lemma~\ref{lem:cadet-condition-simple} ensures $\{x_k-x_\ell=t\}$ is in $\wA_L$ as wanted, unless $|B_k|=|B_\ell|=1$. 
Finally, in the case $|B_k|=|B_\ell|=1$ one has $B_k=\{a_k\}$, $B_\ell=\{j\}$ with $a_k<j$, which implies $k<\ell$ because of the convention $\min(B_1)<\cdots <\min(B_d)$. This, together with $t=s=0$, contradicts the assumption that $(k,\ell,t)$ is in $\tr_{mn}^d$ (hence there is nothing to prove in this case).
This completes the proof that $\wt T$ is in $\Tmnd(\wA_L)$. Thus, $\Gamma(T,\mu)$ is in $\wTmA$. 
\end{proof}

It is not hard to see that the map $\Gamma:\oTmn(\mA)\to \wTmA$ is injective. Indeed, given $\Gamma(T,\mu)=(L,\wt T)$, one can construct the function $\de_L$ which specifies how to ``dilate'' each node of $\wt T$ into a path of marked cadet edges. This is represented in Figure~\ref{fig:condition-tree-proof}(b). We will now describe this process more formally, thereby defining a map $\Upsilon$ which is the inverse of $\Gamma$.

Let $(L,\wt T)\in \wTmA$. Let $(\{B_1,\ldots,B_d\},\de_L)=\Lambda(L)$, with our usual convention $\min(B_1)<\min(B_2)<\cdots<\min(B_d)$. We want to replace each node $k\in[d]$ of $\wt T$ by a path of marked cadet edges with vertices in $B_k$, so as to produce a marked tree $(T,\mu)$ in $\oTmn$.
Recall that $\prec_L$ is the total order on $B_k$ such that $i\precL j$ if $\de_L(i)<\de_L(j)$ or ($\de_L(i)=\de_L(j)$ and $i<j$). Let $j_1\prec_L j_2\prec_L \cdots \precL j_q$ be the elements of $B_k$, and let $s_p=\de_L(j_{p+1})-\de_L(j_p)$ for all $p\in [q-1]$. 
Let 
$$\de^k~:=~\max(\de_L(j)\mid j\in B_k)~=~\de_L(j_q)~=~s_1+\cdots +s_{q-1}.$$
It is easy to see that $\wA_L$ has no hyperplane of the form $\{x_k-x_{\ell}=t\}$ for $t>\de^k+m$. Since  $\wt T$ is in $\Tmnd(\wA_L)$, this implies that the $t$-child of the node $k$ in $\wt T$ is a leaf for all $t>\de^k+m$.
Hence, one can delete the $t$-child of the node $k$ of $\wT$ for all $t>\de^k+m$ (these are all leaves), and then replace the node $k$ by a path $P_k$ of made of the vertices $j_1,\ldots,j_q$ with $j_{p+1}=s_p\cadet(j_p)$ for all $p\in[q-1]$, and add $m-s_p$ leaves as right children of each node of $j_p$ so that every node in $P_k$ has $m+1$ children. We declare all the cadet edges of the path $P_k$ as marked (for all $k\in[d]$), and denote by $\Ups(L,T)=(T,\mu)\in \oTmn$ the marked tree obtained in this manner.

\begin{lemma} \label{lem:uncontraction-tree}
For any pair $(L,\wt T)$ in $\wTmA$, the marked tree $(T,\mu)=\Ups(L,T)$ is in $\oTmn(\mA)$.
\end{lemma}

\begin{proof}
Let $(L,\wt T)\in \wTmA$, let $(\{B_1,\ldots,B_d\},\de_L)=\Lambda(L)$, and let $(T,\mu)=\Ups(L,T)$. 
By definition of $\Ups$, one has $\blocks(\mu)=\{B_1,\ldots,B_d\}$. By Lemma~\ref{lem:code-subspace}, the pair $(\{B_1,\ldots,B_d\},\de_L)$ is in $\mP(\mA)$. This implies that each pair $(B_k,\de_L)$ is $\mA$-connected (see Definition~\ref{def:A-connected-partition}), hence each block $B_k$ of $\blocks(\mu)$ is $\mA$-connected (see Definition~\ref{def:A-connected-tree}). Thus, $(T,\mu)$ is $\mA$-connected.

It remains to show that $(T,\mu)$ satisfies the $\mA$-cadet condition.
Let $e=\{i,j\}$ be a non-marked cadet edge of $(T,\mu)$, with $i=s\cadet(j)$. We need to show that either ($s=0$ and $i<j$) or (there exists $i'\musim i$ and $j'\musim j$ such that the hyperplane $\{x_{i'}-x_{j'}=\driftT(i')-\driftT(j')\}$ is in $\mA$).
One has $i\in B_k$ and $j\in B_\ell$ for some $k,\ell\in [d]$, with $i$ the ancestor in $T$ of all the vertices in $B_k$, and $j$ the descendant in $T$ of all the vertices in $B_\ell$. Moreover, in $\wt T$ one has $k=t\cadet(\ell)$ with $t=s+\de^\ell$, where $\de^\ell=\de_L(j)=\max(\de_L(j')\mid j'\in B_\ell)$.
Since $\wt T$ is in $\Tmnd(\wA_L)$, one of the following situations holds:
\bitem
\item[(a)] $\{x_{k}-x_{\ell}=t\}$ is in $\wA_L$, or
\item[(b)] $t=0$ and $k<\ell$ (so that $(k,\ell,t)$ is not in $\tr_{mn}^n$). 
\eitem
%In situation (b) one has $s=0$ and $i<j$ as wanted (WARNING: Need to assume convention $a_1<a_2<...<a_k$). 
In situation (a), by definition of $\wA_L$, there is a hyperplane of the form $\{x_{i'}-x_{j'}=s'\}$ in $\mA$ such that $i'\in B_k$, $j'\in B_\ell$ and $t=s'-\de_L(i')+\de_L(j')$. 
Solving for $s'$ we get
$$s'=t+\de_L(i')-\de_L(j')=(s+\de^\ell)+(\driftT(i')-\driftT(a_k))-(\driftT(j')-\driftT(a_\ell))=\driftT(i')-\driftT(j'),$$
since $\driftT(a_k)-\driftT(a_\ell)=s+\de^\ell$. Hence we have found $i'\musim i$ and $j'\musim j$ such that the hyperplane $\{x_{i'}-x_{j'}=\driftT(i')-\driftT(j')\}$ is in $\mA$, as wanted.
If (a) does not hold and (b) holds, then Lemma~\ref{lem:cadet-condition-simple} implies that $|B_k|=|B_\ell|=1$. Therefore $i<j$ by our convention 
$\min(B_1)<\min(B_2)<\cdots<\min(B_d)$. Hence $s=0$ and $i<j$ as wanted.
\end{proof}

It is not hard to see that the maps $\Gamma:\oTmn(\mA)\to \wTmA$ and $\Ups:\wTmA\to\oTmn(\mA)$ satisfy $\Ga\circ\Ups=\Id$ and $\Ups\circ\Ga=\Id$. Thus,  these maps are bijections between $\oTmn(\mA)$ and $\wTmA$.

In order to conclude the proof of Theorem~\ref{thm:main}, it now suffices to show that the diagram in Figure~\ref{fig:diagramproof} is commutative. In other words, we need to show that $\bPhi_\mA=\Theta\circ \wt \Phi_\mA\circ \Gamma$. Consider a marked tree $(T,\mu)\in \oTmn(\mA)$, and let $(L,\wt T)=\Gamma(T,\mu)$. As observed above, one has
$$L=\bigcap_{\substack{\{i,j\}\in \mu \\ i=s\child(j)}}\{x_i-x_j=s\},$$
hence 
$$\bPhi_\mA(T,\mu)=L\cap\left(\bigcap_{\substack{(i,j,s)\in \trmn(\mA)\\ i\nmusim j \\ i\precT s\child(j)}}\{x_i-x_j< s\}\right)\cap \left(\bigcap_{\substack{(i,j,s)\in \trmn(\mA)\\i\nmusim j \\ i\succeqT s\child(j)}}\{x_i-x_j>s\}\right).$$
\newcommand{\trAL}{\tr_{mn}^d(\wA_L)}
We need to compare this expression to 
\begin{eqnarray*}
\Theta\circ \wt \Phi_\mA(L,\wt T)\!\!\!&=&\!\!\!
\psi_L^{-1}\left(\left(\bigcap_{\substack{(k,\ell,t)\in \trAL\\ k\precwT t\child(\ell)}}\{x_k-x_\ell< t\}\right)\cap\left(\bigcap_{\substack{(k,\ell,t)\in \trAL\\ k\succeqwT t\child(\ell)}}\{x_k-x_\ell> t\} \right)\right)\\
&=&\!\!\!\left(\bigcap_{\substack{(k,\ell,t)\in \trAL\\ k\precwT t\child(\ell)}}\!\!\!\!\!\psi_L^{-1}(\{x_k-x_\ell< t\})\right)\cap\left(\bigcap_{\substack{(k,\ell,t)\in \trAL\\ k\succeqwT t\child(\ell)}}\!\!\!\!\!\psi_L^{-1}(\{x_k-x_\ell> t\}) \right),
\end{eqnarray*}
where the second equality holds because $\psi_L:L\to \RR^d$ is bijective.

By definition of $\wA_L$, a triple $(k,\ell,t)\in\tr_{mn}^d$ is in $\trAL$ if and only if there exists a hyperplane $\{x_i-x_j=s\}$ in $\mA$ such that $i\in B_k$, $j\in B_\ell$ and $t=s-\de_L(i)+\de_L(j)$. In this situation we say that $(i,j,s)$ is a \emph{representative} of $(k,\ell,t)$.

It is easy to check that if $(i,j,s)$ is a representative of a triple $(k,\ell,t)\in\trAL$ then
$$\psi_L^{-1}(\{x_k-x_\ell< t\})=L\cap\{x_i-x_j<s\}\textrm{ and }\psi_L^{-1}(\{x_k-x_\ell> t\})=L\cap\{x_i-x_j>s\}.$$
Indeed, the first identity follows from observing that 
$$\psi_L^{-1}(\{x_k-x_\ell< t\})=L\cap\{(x_1,\ldots,x_n)\in \RR^n\mid x_{a_k}-x_{a_\ell}<s-\de_L(i)+\de_L(j)\},$$ where $a_k$ and $a_\ell$ are the minimum element of $B_k$ and $B_\ell$ respectively (for the $\precL$ order), and $x_i=x_{a_k}+\de_L(i)$, and $x_j=x_{a_\ell}+\de_L(j)$. The second identity is proved identically.

We call a representative $(i,j,s)$ of $(k,\ell,t)\in\trAL$ a \emph{forward} representative if $(i,j,s)$ is in $\trmn(\mA)$ and a \emph{backward} representative if $(j,i,-s)$ is in $\trmn(\mA)$. Note that any representative is either forward or backward (but not both).
Hence proving that $\bPhi_\mA(T,\mu)=\Theta\circ \wt \Phi_\mA(L,\wt T)$ amounts to showing that the following holds for any triple $(k,\ell,t)\in \trAL$:
\bitem
\item[(a)] For any forward representative $(i,j,s)$ of $(k,\ell,t)$, one has $k\precwT w$ if and only if $i\precT v$, where $w=t\child(\ell)$ (in $\wt T$) and $v=s\child(j)$ (in $T$). 
\item[(b)] For any backward representative $(i,j,s)$ of $(k,\ell,t)$, one has $k\precwT w$ if and only if $j\succeqT v$, where $w=t\child(\ell)$ (in $\wt T$) and $v=(-s)\child(i)$ (in $T$). 
\eitem

We proceed to prove Properties (a) and (b). We start with (a).
Let $(i,j,s)$ be a forward representative of $(k,\ell,t)\in \trAL$. Let $w=t\child(k)$ and $v=s\child(j)$. 
%Recall that the orders $\precT$ and $\precwT$ involve some ``path'' and ``drift'' functions. % which we will index by the relevant trees for clarity. 
Observe that 
\begin{equation}\label{eq:drift}
\drift_T(i)=\drift_\wT(k)+\de_L(i)\textrm{ and }\drift_T(j)=\drift_\wT(\ell)+\de_L(j).
\end{equation}
Combining~\eqref{eq:drift} with $\drift_T(v)=s+\drift_T(j)$, $\drift_\wT(w)=t+\drift_\wT(\ell)$ and $t=s-\de_L(i)+\de_L(j)$ gives
$$\driftT(i)-\driftT(v)=\drift_\wT(k)-\drift_\wT(w).$$ 
%$\drift_T(i)< \drift_T(v)$ (resp. $\drift_T(i)> \drift_T(v)$, $\drift_T(i)= \drift_T(v)$) if and only if $\drift_\wT(k)< \drift_\wT(w)$ (resp. $\drift_\wT(k)> \drift_\wT(w)$, $\drift_\wT(k)= \drift_\wT(w)$). 
This proves Property (a) except in the case where $\drift_T(i)= \drift_T(v)$ and $\drift_\wT(k)= \drift_\wT(w)$. In this remaining case, we need to show that $\path_T(i)\prec \path_T(v)$ if and only if $\path_\wT(k)\prec \path_\wT(w)$, 
where the order $\prec$ on finite sequence of integers is defined by $(s_1,\ldots,s_a)\prec (s_1',\ldots,s_{b}')$ if and only if there exists $c\leq a$ such that $(s_1,\ldots,s_c)=(s_1',\ldots,s_c')$ and ($c=a$ or $s_{c+1}>s'_{c+1}$). This amounts to an easy but tedious check, which is illustrated in Figure~\ref{fig:proof-treeorder1}. 
Note that $\path_T(i)\prec \path_T(v)$ means that either $i$ is a strict ancestor of $v$ in $T$, or that the path from the root vertex to $i$ in $T$ is strictly to the right of the path from the root vertex to $v$. These situations are represented on the left of Figure~\ref{fig:proof-treeorder1}, and it easy to see that in those cases $\path_\wT(k)\prec \path_\wT(w)$. The situations where $\path_T(i)\succeq \path_T(v)$ are represented on the right of Figure~\ref{fig:proof-treeorder1}, and it easy to see that in those cases $\path_\wT(k)\succeq \path_\wT(w)$. This concludes the proof of Property (a).

\fig{width=\linewidth}{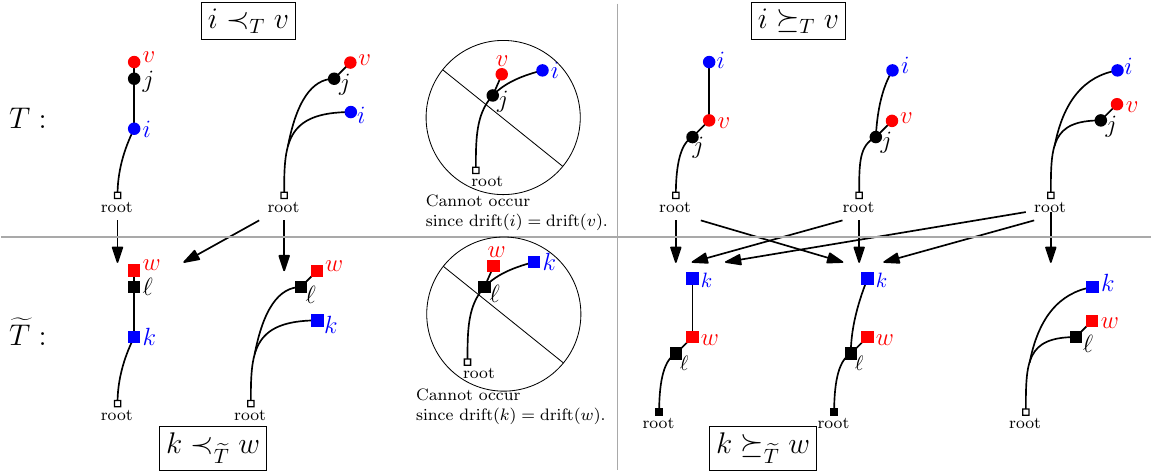}{Proof of Property (a) in the case where $\drift_T(i)= \drift_T(v)$ and $\drift_\wT(k)= \drift_\wT(w)$. The possible configurations for the tree $T$ are represented in the top row, while the possible configurations for the tree $\wT$ are represented in the bottom row. The arrows indicate the possible transformations from $T$ to $\wT$. Recall that the path of marked edges of $T$ containing $i$ (resp. $j$) becomes the node $k$ (resp. $\ell$) of $\wT$.
Left: The configurations such that $\path_T(i)\prec \path_T(v)$, and the corresponding configurations in $\wT$ (indicated by arrows). Right: The configurations such that $\path_T(i)\succeq \path_T(v)$, and the corresponding configurations in $\wT$ (indicated by arrows).}

The proof of Property (b) is similar. Let $(i,j,s)$ be a backward representative of $(k,\ell,t)\in \trAL$. Let $w=t\child(k)$ and $v=(-s)\child(i)$. Combining~\eqref{eq:drift} with $\drift_T(v)=-s+\drift_T(i)$, $\drift_\wT(w)=t+\drift_\wT(\ell)$ and $t=s-\de_L(i)+\de_L(j)$ shows that 
$$ \driftT(j)-\driftT(v)=\drift_\wT(w)-\drift_\wT(k).$$
%$\drift_T(j)< \drift_T(v)$ (resp. $\drift_T(j)> \drift_T(v)$, $\drift_T(j)= \drift_T(v)$) if and only if $\drift_\wT(k)>\drift_\wT(w)$ (resp. $\drift_\wT(k)< \drift_\wT(w)$, $\drift_\wT(k)= \drift_\wT(w)$). 
This proves Property (b) except in the case where $\drift_T(j)= \drift_T(v)$ and $\drift_\wT(k)= \drift_\wT(w)$. In this remaining case, we need to show $\path_T(j)\prec \path_T(v)$ if and only if $\path_\wT(k)\succeq \path_\wT(w)$. This equivalence is illustrated in Figure~\ref{fig:proof-treeorder2}.% and is left to the reader.
\fig{width=\linewidth}{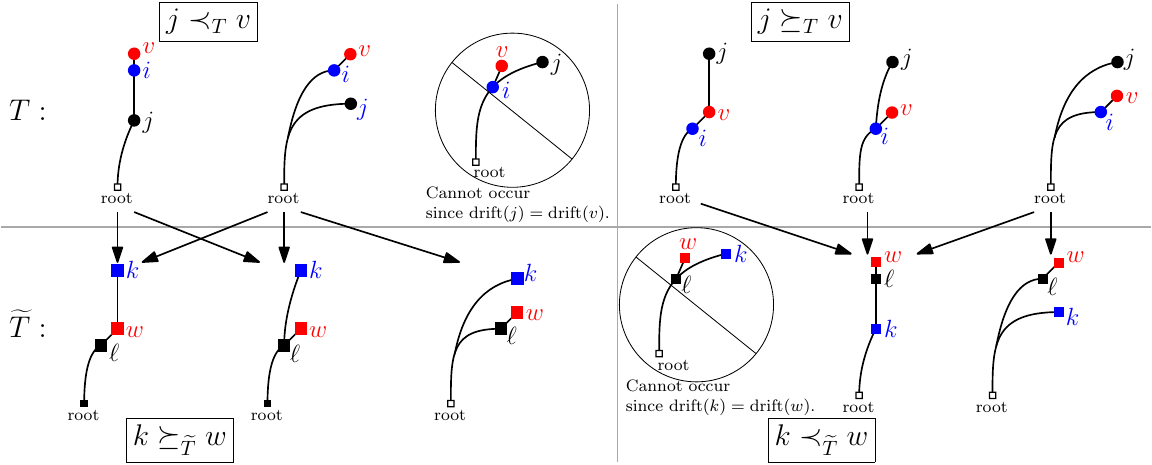}{Proof of Property (b) in the case where $\drift_T(j)= \drift_T(v)$ and $\drift_\wT(k)= \drift_\wT(w)$. 
The possible configurations for the tree $T$ are represented in the top row, while the possible configurations for the tree $\wT$ are represented in the bottom row.
Left: The configurations such that $\path_T(j)\prec \path_T(v)$, and the corresponding configurations in $\wT$. Right: The configurations such that $\path_T(j)\succeq \path_T(v)$, and the corresponding configurations in $\wT$.}

This concludes the proof of Properties (a) and (b) (thereby showing that $\bPhi_\mA=\Theta\circ \wt \Phi_\mA\circ \Gamma$), and completes the proof of Theorem~\ref{thm:main}.

%\section{Extension to infinite arrangements}

%\bigskip
%\noindent{\bf Conflict of interest statement:} The authors have no conflict of interest to disclose in relation to this article.\\

\smallskip
\noindent{\bf Acknowledgments:} This work was supported by NSF Grant DMS-2154242.

\bibliographystyle{alpha}
\bibliography{biblio-hyperplanes}

\end{document}

%% file: defs.tex
\usepackage{amssymb,amsmath}
\usepackage{stmaryrd,paralist}
\usepackage{bbm}
\usepackage{graphics,graphicx,tikz-cd}

\usepackage{subfigure}
\usepackage[margin=1.4 in]{geometry}
\usepackage{hyperref,color}
\graphicspath{{Figures/}} 
\graphicspath{{./Figures/}}

\usepackage{url}

\newtheorem{theo}{Theorem}
\numberwithin{theo}{section}
\newtheorem{thm}[theo]{Theorem}

\newtheorem{lemma}[theo]{Lemma}
\newtheorem{prop}[theo]{Proposition}

\newtheorem{definition}[theo]{Definition}
\newtheorem{notation}[theo]{Notation}
\newtheorem{cor}[theo]{Corollary}
\theoremstyle{remark}
\newtheorem{remark}[theo]{Remark}
\newtheorem{example}[theo]{Example}
\newtheorem{rk}[theo]{Remark}

{\medbreak}

\newcommand{\ds}{\displaystyle}
\def\ni{\noindent}

\newcommand{\bitem}{\begin{compactitem}}
\newcommand{\eitem}{\end{compactitem}}

\def\ZZ{\mathbb{Z}}
\def\NN{\mathbb{N}}
\def\QQ{\mathbb{Q}}
\def\RR{\mathbb{R}}

\def\mA{\mathcal{A}}
\def\mB{\mathcal{B}}
\def\mC{\mathcal{C}}
\def\mD{\mathcal{D}}
\def\mE{\mathcal{E}}
\def\mF{\mathcal{F}}
\def\mG{\mathcal{G}}
\def\mH{\mathcal{H}}
\def\mL{\mathcal{L}}

\def\mP{\mathcal{P}}
\def\mR{\mathcal{R}}

\def\mT{\mathcal{T}}

\def\mShi{\mathcal{S}hi}

\newcommand{\Id}{\textrm{Id}}

\newcommand{\de}{\delta}

\newcommand{\Om}{\Omega}
\newcommand{\Ga}{\Gamma}
\newcommand{\De}{\Delta}

\newcommand{\Ups}{\Upsilon}

\newcommand{\ov}[1]{\overline{#1}}
\newcommand{\wt}[1]{\widetilde{#1}}

\newcommand{\bPhi}{\ov{\Phi}}

\newcommand{\fig}[3]{\begin{figure}[h!]\begin{center}\includegraphics[#1]{#2.pdf}\end{center}\caption{#3}\label{fig:#2}\end{figure}}

\newcommand{\mm}{\textbf{m}}

\newcommand{\fS}{\mathfrak{S}} % Symmetric group
\DeclareMathOperator{\one}{\textbf{1}}

\DeclareMathOperator{\cadet}{\!\textrm{-cadet}}
\DeclareMathOperator{\child}{\!\textrm{-child}}
\renewcommand{\path}{\textrm{path}}
\DeclareMathOperator{\drift}{\textrm{drift}}
\newcommand{\pathT}{\path_T}
\DeclareMathOperator{\driftT}{\drift_T}
\DeclareMathOperator{\blocks}{\textrm{blocks}}
\DeclareMathOperator{\shadow}{\textrm{shadow}}
\DeclareMathOperator{\start}{\textrm{start}}
\DeclareMathOperator{\eend}{\textrm{end}}
\DeclareMathOperator{\comp}{\textrm{comp}}

\DeclareMathOperator{\tr}{\textrm{Triple}}
\DeclareMathOperator{\trmn}{\tr_m^n}

\newcommand{\wT}{{\wt{T}}}
\newcommand{\vD}{{\vec{\mD}}}
\newcommand{\oV}{{\ov{V}}}
\newcommand{\oD}{{\ov{D}}}
\newcommand{\oE}{{\ov{E}}}
\newcommand{\oG}{{\ov{\mG}}}

\DeclareMathOperator{\Tmn}{\mT_m^n}
\DeclareMathOperator{\Tmnd}{\mT_{mn}^d}
\DeclareMathOperator{\oT}{\ov{\mT}}
\DeclareMathOperator{\oTmn}{\ov{\mT}_m^n}
\DeclareMathOperator{\wA}{\wt{\mA}}

\DeclareMathOperator{\precT}{\prec_T}
\DeclareMathOperator{\precwT}{\prec_{\wt T}}
\DeclareMathOperator{\preceqT}{\preceq_T}
\DeclareMathOperator{\succeqT}{\succeq_T}
\DeclareMathOperator{\succeqwT}{\succeq_{\wt T}}
\DeclareMathOperator{\precL}{\prec_L}
\DeclareMathOperator{\preceqL}{\preceq_L}

\newcommand{\musim}{\overset{\mu}{\sim}}
\newcommand{\nmusim}{\overset{\mu}{\nsim}}
\newcommand{\Lsim}{\overset{L}{\sim}}

\newcommand{\XX}{e^{x}-1}